\documentclass[a4paper, 11pt]{amsart}

\usepackage[margin=1in,marginparwidth=0.8in, marginparsep=0.1in]{geometry}

\usepackage{graphicx}
\usepackage{marginnote}

\usepackage[british,UKenglish,USenglish,english,american]{babel}

\usepackage{newlfont}
\usepackage{amssymb}
\usepackage{amsmath,mathtools}
\usepackage{amsfonts}
\usepackage{latexsym}
\usepackage{amsthm}
\usepackage{mathrsfs}
\usepackage{stmaryrd}
\usepackage[all,cmtip]{xy}
\usepackage{caption}
\usepackage{enumerate}

\UseRawInputEncoding

\usepackage{amsmath,amssymb,amscd,xypic}
\usepackage{amsfonts}
\usepackage[leqno]{amsmath}
\usepackage{mathrsfs} 
\usepackage{enumerate}

\usepackage[usenames,dvipsnames]{xcolor}

\usepackage{layout}
\usepackage{fullpage}

\usepackage[backref=page]{hyperref}

\hypersetup{
 colorlinks,
 citecolor=Green,
 linkcolor=blue,
 urlcolor=Blue}

\theoremstyle{plain}                    
\newtheorem{teo}{Theorem}[section]     
\newtheorem{theoremalpha}{Theorem}

\newtheorem{coralpha}[theoremalpha]{Corollary}
\newtheorem{prop}[teo]{Proposition}

\newtheorem{fact}[teo]{Fact}
\newtheorem{cor}[teo]{Corollary}       
\newtheorem{lem}[teo]{Lemma}            
\newtheorem{deflem}[teo]{Definition/Lemma} 
\theoremstyle{definition}               
\newtheorem{notations}[teo]{}

\newtheorem{defin}[teo]{Definition}
\newtheorem{ese}[teo]{Example} 

\newtheorem{deflemma}[teo]{Definition-Lemma}
\theoremstyle{remark}
\newtheorem{rmk}[teo]{Remark}

\newenvironment{sis}{\left\{\begin{aligned}}{\end{aligned}\right.}

\newcommand{\Filippo}[1]{{\color{green}{Filippo: #1}}}

\newcommand{\bbC}{{\mathbb C}}
\newcommand{\bbQ}{{\mathbb Q}}
\newcommand{\bbZ}{{\mathbb Z}}
\newcommand{\bbP}{{\mathbb P}}
\newcommand{\bbG}{{\mathbb G}}

\newcommand{\bbE}{{\mathbb E}}
\newcommand{\bbF}{{\mathbb F}}

\newcommand{\cC}{{\mathcal C}}

\newcommand{\cE}{{\mathcal E}}

\newcommand{\cJ}{{\mathcal J}}
\renewcommand{\cL}{{\mathcal L}}

\newcommand{\cM}{{\mathcal M}}
\newcommand{\cN}{{\mathcal N}}
\newcommand{\cO}{{\mathcal O}}
\newcommand{\cP}{{\mathcal P}}
\newcommand{\cQ}{{\mathcal Q}}

\newcommand{\un}{\underline}
\newcommand{\ov}{\overline}
\newcommand{\wt}{\widetilde}
\newcommand{\wh}{\widehat}

\renewcommand{\rm}{\mathrm}
\newcommand{\scr}{\mathscr}

\numberwithin{equation}{section}

\DeclareMathOperator{\Pic}{Pic}

\DeclareMathOperator{\Aut}{Aut}
\DeclareMathOperator{\Spec}{Spec}

\DeclareMathOperator{\Br}{Br}

\DeclareMathOperator{\RPic}{RPic}
\DeclareMathOperator{\NS}{NS}
\DeclareMathOperator{\Hom}{Hom}

\DeclareMathOperator{\taut}{taut}
\DeclareMathOperator{\coker}{coker}
\renewcommand{\Im}{\text{Im}}

\DeclareMathOperator{\End}{End}
\DeclareMathOperator{\id}{id}
\DeclareMathOperator{\res}{res}
\DeclareMathOperator{\red}{red}

\DeclareMathOperator{\rk}{rk}

\DeclareMathOperator{\Gm}{\mathbb{G}_{m}}

\DeclareMathOperator{\PSL}{PSL}
\DeclareMathOperator{\GL}{GL}
\DeclareMathOperator{\SL}{SL}
\renewcommand{\div}{\mathrm{div}}

\DeclareMathOperator{\SO}{SO}

\DeclareMathOperator{\PSO}{PSO}
\DeclareMathOperator{\Spin}{Spin}

\DeclareMathOperator{\Sp}{Sp}

\DeclareMathOperator{\PSp}{PSp}

\DeclareMathOperator{\ord}{ord}

\DeclareMathOperator{\Bun}{\mathrm{Bun}}

\newcommand{\bgr}[1]{\mathfrak{Bun}_{#1,g,n}}
\DeclareMathOperator{\Bunr}{\mathfrak{Bun}}

\DeclareMathOperator{\Bil}{Bil}

\DeclareMathOperator{\Sym}{Sym}
\DeclareMathOperator{\ev}{ev}

\renewcommand{\ss}{\mathrm{ss}}
\newcommand{\ab}{\mathrm{ab}}

\newcommand{\ad}{\mathrm{ad}}
\renewcommand{\sc}{\mathrm{sc}}

\newcommand{\g}{\mathfrak g}

\newcommand{\roo}{\mathrm{roots}}
\newcommand{\coroo}{\mathrm{coroots}}
\newcommand{\wei}{\mathrm{weights}}
\newcommand{\cowei}{\mathrm{coweights}}

\DeclareMathOperator{\w}{wt}
\DeclareMathOperator{\obs}{obs}

\title{The line bundles on the moduli stack of principal bundles on families of curves}

\author{Roberto Fringuelli}
\address{Roberto Fringuelli, Dipartimento di Matematica, Universit\`a di Roma ``La Sapienza'',  Piazzale Aldo Moro 5, 00185 Roma, Italy}
\email{r.fringuelli@uniroma1.it}

\author{Filippo Viviani}
\address{Filippo Viviani, Dipartimento di Matematica, Universit\`a di Roma ``Tor Vergata'', Via della Ricerca Scientifica 1, 00133 Roma, Italy}
\email{viviani@mat.uniroma2.it}

\begin{document}

\begin{abstract}
Given a connected reductive algebraic group $G$ over an algebraically closed field, we investigate the Picard group of the moduli stack of principal $G$-bundles over an arbitrary family of smooth curves. 
\end{abstract}

\maketitle

\tableofcontents

\section{Introduction}

The aim of this paper, which is a sequel of the papers \cite{FV1} and \cite{FV2}, is to study 
 the Picard group of the \emph{moduli stack of (principal) $G$-bundles $\Bun_G(C/S)$}, which parametrizes  $G$-bundles, where $G$ is a connected and smooth linear algebraic group  over a field $k=\ov k$,  over \emph{an arbitrary family} $\pi: C\to S$ of (connected, smooth and projective) $k$-curves of genus $g=g(C/S)\geq 0$. 
 
 \vspace{0.2cm}
 
 The moduli stack $\Bun_G(C)$ of (principal) $G$-bundles, where $G$ is a complex reductive group and $C$ is a complex curve, has been deeply studied because of its relation to the Wess-Zumino-Witten (=WZW) model associated to $G$, which form a special class of rational conformal field theories, see \cite{Bea94}, \cite{Sor96} and \cite{Bea96} for nice surveys. In the WZW-model associated to a simply connected group $G$, the spaces of conformal blocks can be interpreted  as  spaces of  generalized theta functions, that is spaces of global sections of suitable line bundles (e.g. powers of determinant line bundles) on $\Bun_G(C)$, see \cite{BL, KNR94, Fal94, Pau96, LS97}. 
The above application to conformal field theory leads naturally to the study of the Picard group of $\Bun_G(C)$ (for $G$ reductive over $k=\ov k$), which was computed in increasing degree of generality thanks to the effort of many mathematicians:  the case of $\SL_r$ dates back to pioneering work of Drezet-Narasimhan in the late eighties \cite{DN}; the case of a simply connected, almost-simple $G$ is dealt with in \cite{KN97, LS97, SO99, Fa03, BK05}; the case of a semisimple almost-simple $G$ is dealt with in \cite{Laszlo, BLS98, Tel98}; the case of an arbitrary reductive group was finally established by Biswas-Hoffmann \cite{BH10}.

Motived by the long term project of generalizing the study of the vector bundle of conformal blocks (also known as Verlinde bundle) behind the simply connected case, the authors in \cite{FV1} and \cite{FV2} have computed the Picard group of $\Bun_G(\cC_{g,n}/\cM_{g,n})$, for the universal family $\cC_{g,n}\to \cM_{g,n}$ over the moduli stack of $n$-pointed curves of genus $g$, as well as the restriction homomorphism to the Picard group of the fibers of $\Bun_G(\cC_{g,n}/\cM_{g,n})\to \cM_{g,n}$, generalizing the previous works of Melo-Viviani \cite{MV} for $G=\Gm$ and Fringuelli \cite{Fri18} for $G=\GL_r$.

In this follow up paper, we study the Picard group of $\Bun_G(C/S)$ for an arbitrary family $C/S$ of curves over an algebraic stack $S$ (which, for tecnichal reasons, it is often assumed to be an integral and regular quotient stack), in which case very little was known before. A remarkable exception is the work of Faltings \cite{Fa03}, which provides a functorial identification $\Pic \Bun_G(C/S)\cong \bbZ$ for a simply connected and almost simple group $G$ and an arbitrary family $C/S$ admitting a section. 
Our Theorem \ref{Main:PicG>0} can be seen as a generalization of Faltings' result (see Corollary \ref{MCor:PicG>0} and the discussion following it).

\vspace{0.1cm}

We can now explain our main results focusing, in this introduction, on families $C/S$ of positive genus $g(C/S)>0$, and referring to the body of the paper for the sligthly different, and easier, case $g(C/S)=0$.

Recall that the stack $\Bun_{G}(C/S)$ is an algebraic stack, locally of finite type and smooth over $S$ and  its relative connected  components are in functorial bijection with the fundamental group $\pi_1(G)$ (see Fact \ref{F:propBunG}). We will denote the connected components and the restriction of the forgetful morphism by 
$$\Phi_G^{\delta}: \Bun_{G}^{\delta}(C/S)\to S \quad \text{ for any } \delta\in \pi_1(G). $$
The pull-back morphism at the level of the Picard groups is injective, so will often focus onto the relative Picard group
$$
\RPic \Bun_G^{\delta}(C/S):=\frac{\Pic\Bun_G^{\delta}(C/S)}{(\Phi^{\delta})^*(\Pic S)}.
$$


We proved in \cite[Thm. A]{FV1} that if $\red: G\twoheadrightarrow G^{\red}$ is the  reductive quotient of $G$, i.e. the quotient of $G$ by its unipotent radical, then for any $\delta\in \pi_1(G)\xrightarrow[\cong]{\pi_1(\red)}\pi_1(G^{\red})$ the pull-back homomorphism 
$$
\red_\#^*:\Pic(\Bun_{G^{\red}(C/S)}^{\delta})\xrightarrow{\cong} \Pic(\Bun_{G}^{\delta}(C/S)) 
$$
is an isomorphism.   Hence, throughout this paper, we will restrict to the case of a \emph{reductive group} $G$. 

\vspace{0.2cm}

A first source of line bundles on $\Bun_G^{\delta}(C/S)$ comes from the determinant of cohomology $d_{\wh \pi}(-)$ and the Deligne pairing $\langle -,-\rangle_{\wh \pi}$ of line bundles on the universal curve $\wh \pi:C\times_S \Bun_G^{\delta}(C/S)\to\Bun_G^{\delta}(C/S)$.
To be more precise,  any character $\chi\in \Lambda^*(G):=\Hom(G, \Gm)$ gives rise to a morphism of $S$-stacks 
$$
\chi_\#:\Bun_G^{\delta}(C/S)\to\Bun_{\Gm}^{\chi(d)}(C/S)
$$
and, by  pulling back via $\chi_{\#}$ the universal $\Gm$-bundle (i.e. line bundle) on the universal curve $C\times_S \Bun_{\Gm}^{\chi(\delta)}(C/S)\to\Bun_{\Gm}^{\chi(\delta)}(C/S)$, we get a line bundle $\mathcal L_{\chi}$ on  $C\times_S \Bun_G^{\delta}(C/S)$. Then, using these line bundles $\mathcal L_{\chi}$  and the line bundles $\{p_1^*(M)\::\: M\in \Pic(C)\}$ pull-backed from the family $C/S$, we define the following line bundles, that we call \emph{tautological line bundles}, on $\Bun_G^{\delta}(C/S)$
\begin{equation*}\label{E:tautINT}
\begin{aligned}
&\left\{d_{\wh \pi}(\cL_{\chi}(M))\: : \chi\in \Lambda^*(G), M\in \Pic(C)\right\},\\ 
& \left\{\langle \cL_{\chi}(M),\cL_{\mu}(N)\rangle_{\wh \pi}: \chi,\mu\in \Lambda^*(G), M,N\in \Pic(C) \right\},\\
\end{aligned}
\end{equation*}
where we set $\cL_{\chi}(M):=\cL_{\chi}(p_1^*(M))$ for simplicity. 

Note that since any character of $G$ factors through the abelianization of $\ab:G\to G^{\ab}$ of $G$, 
then the tautological line bundles lie in the subgroup 
$$
\Pic \Bun_{G^{\ab}}^{\delta^{\ab}}(C/S)\stackrel{\ab_{\sharp}^*}{\hookrightarrow} \Pic \Bun_{G}^{\delta}(C/S),
$$
where $\delta^{\ab}:=\pi_1(\ab)(\delta)$. 

The next Theorem describes, for an arbitrary torus $T$ and an arbitrary $d\in \pi_1(T)$ (which includes the case $T=G^{\ab}$ and $\delta^{\ab}=d$), the structure of the
\emph{tautological subgroup} 
$$\Pic^{\taut} \Bun_{T}^{d}(C/S)\subseteq \Pic \Bun_{T}^{d}(C/S)$$
 generated by the tautological line bundles (or, in other words, it encodes all the relations among tautological line bundles), and it describes when  it coincides with the full Picard group.

\begin{theoremalpha}\label{Main:PicT>0}(see Theorems \ref{T:PicBunT>0}, \ref{T:all-taut})
Let $C\to S$ be a family of curves of genus $g>0$ on a regular and integral quotient stack. Let $T$ be a torus  and fix $d\in \Lambda(T)=\pi_1(T)$.
\begin{enumerate}
\item \label{Main:PicT>01} The relative tautological  Picard group of $\Bun_{T}^{d}(C/S)$ sits in the following exact sequence,  functorial in $S$ and in $T$:
\begin{equation*}
0 \to  \Lambda^*(T)\otimes \RPic(C/S)  \xrightarrow{i_{T}^{d}}  \RPic^{\taut}\Bun_{T}^{d}(C/S) \xrightarrow{\gamma_{T}^{d}}  \Bil^s(\Lambda(T)) \to 0 
\end{equation*}
where $\RPic(C/S)$ is the relative Picard group of $C/S$, and  the homomorphisms $i_{T}^{d}$ and $\gamma_{T}^{d}$ are defined by 
$$i_{T}^{d}(\chi\otimes M)=\langle \cL_{\chi}, M\rangle_{\wh \pi} \quad \text{ and } \quad \gamma_{T}^{d}(d_{\wh \pi}(\cL_{\chi}(M))=\chi\otimes \chi.$$
\item  \label{Main:PicT>02}   If $\End(J_{C_{\ov \eta}}) =\bbZ$ for a geometric generic point $\ov \eta\to S$ and the natural map $\Pic(C)\rightarrow \Pic_{C/S}(S)$ is surjective, then 
$$
\Pic\Bun_{T}^{d}(C/S)=\Pic^{\taut}\Bun_{T}^{d}(C/S).
$$
\end{enumerate}
\end{theoremalpha}

Moreover, in Theorem \ref{T:PicBunT>0}, we give two alternative presentations of the relative tautological Picard group of $\Bun_{T}^{d}(C/S)$. 

The case of families of genus zero curves is easier and it is dealt with in Theorem \ref{T:PicBunT-0}. In particular, we will show that the Picard group of $\Bun_T^d(C/S)$ is always generated by tautological line bundles if $g(C/S)=0$.

The above Theorem \ref{Main:PicT>0} was known:
\begin{itemize}
\item for a curve $C$ over $k=\ov k$ by Biswas-Hoffmann \cite{BH10} (see also \cite[Prop. 4.1.1]{FV1}), in which case the  hypothesis that $\End(J_{C_{\ov \eta}}) =\bbZ$ is also necessary for part \eqref{T:PicBunT>02} to hold (while the second assumption is always satisfied);
\item for the universal family $\cC_{g,n}/\cM_{g,n}$ over the stack of $n$-pointed curves of genus $g>0$ (which satisfies the assumption of part \eqref{T:PicBunT>02}, see 
Remarks \ref{R:eqRPic} and \ref{R:NSjump}\eqref{R:NSjump2}),  by work of the authors \cite[Thm. B]{FV1} and \cite[Thm. A]{FV2}, generalizing  previous results for $T=\Gm$, $g\geq 2$ and $n=0$ by Melo-Viviani \cite{MV}.
\end{itemize}

The other source of line bundles on $\Bun_G^{\delta}(C/S)$, for $G$ non abelian, is the \emph{trasgression map} introduced in \cite[Thm. 5.0.1, Rmk. 5.2.2]{FV1}:
\begin{equation*}
\tau_G^{\delta}=\tau^{\delta}_G(C/S):(\Sym^2 \Lambda^*(T_G))^{\mathscr W_G}\rightarrow \Pic \Bun_G^{\delta}(C/S),
\end{equation*}
which is the homomorphism, functorial in $S$ and $G$, uniquely characterized the following commutative diagram 
 \begin{equation*}
 \xymatrix{
(\Sym^2 \Lambda^*(T_G))^{\mathscr W_G}\ar[r]^{\tau^{\delta}_G} \ar@{^{(}->}[d] & \Pic \Bun_G^{\delta}(C/S)\ar@{^{(}->}[d]^{\iota_\sharp^*}, \\
\Sym^2 \Lambda^*(T_G)\ar[r]^(0.4){\tau^d_{T_G}} & \Pic \Bun_{T_G}^{d}(C/S), \\
\chi\cdot \chi' \ar[r] & \langle \cL_{\chi}, \cL_{\chi'}  \rangle_{\wh \pi} 
}
\end{equation*}
where $\iota:T_G\hookrightarrow G$ is a fixed maximal torus of $G$, $d\in \pi_1(T_G)$ is any lift of $\delta\in \pi_1(G)$, $\scr W_G:=\cN(T_G)/T_G$ is the Weyl group of $G$, 
and $(\Sym^2\Lambda^*(T_G))^{\scr W_G}$ is identified with the lattice  $\Bil^{s,\ev}(\Lambda(T_G))^{\scr W_G}$  of  $\scr W_G$-invariant even symmetric bilinear forms on the lattice $\Lambda(T_G)$ of cocharacters of $T_G$ (see Lemma \ref{L:coker-sym} for an explicit description).

The next Theorem describes when the Picard group of $\Bun_G^{\delta}(C/S)$ is generated by the image of the trasgression map and the pull-back of the line bundles on 
$\Bun_{G^{\ab}}^{\delta^{\ab}}(C/S)$.

\begin{theoremalpha}\label{Main:PicG>0}(see Theorems \ref{T:PicBunGAss}, \ref{T:PicBunGtf})
Let $C\to S$ be a family of curves of genus $g>0$ on a regular and integral quotient stack. Let $G$ be a
reductive group and denote by $\ab:G\to G^{\ab}$ its abelianization. 
Fix $\delta\in \pi_1(G)$ and set $\delta^{\ab}:=\pi_1(\ab)(\delta)\in \pi_1(G^{\ab})$.

 If one of the following two hypothesis is satisfied:
\begin{enumerate}[(a)]
\item \label{Ass-a} the derived subgroup $\scr D(G)$ of $G$ is simply connected;
\item \label{Ass-b} $\End(J_{C_{\ov \eta}}) =\bbZ$ for a geometric generic point $\ov \eta\to S$, the natural map $\Pic(C)\rightarrow \Pic_{C/S}(S)$ is surjective and
 $\RPic^0(C/S)$ is torsion-free;
\end{enumerate}
then the following  diagram, which is functorial in $G$ and $S$,
\begin{equation*}
\xymatrix{
 \Sym^2 \Lambda^*(G^{\ab}) \ar@{^{(}->}[r]^{\Sym^2 \Lambda_{\ab}^*} \ar[d]^{\tau_{G^{\ab}}^{\delta^{\ab}}(C/S)}& \left(\Sym^2 \Lambda^*(T_{G})\right)^{\mathscr W_G} \ar[d]^{\tau_{G}^\delta(C/S)}\\
   \Pic \Bun_{G_{\ab}}^{\delta^{\ab}}(C/S)\ar@{^{(}->}[r]^{\ab_{\sharp}^*} &  \Pic \Bun_G^{\delta}(C/S)
}
\end{equation*}
is a push-out diagram. 

In particular, there exists an exact sequence, functorial in $S$ and in $G$,
$$
0 \to \Pic \Bun_{G_{\ab}}^{\delta^{\ab}}(C/S)\xrightarrow{\ab_{\sharp}^*}   \Pic \Bun_G^{\delta}(C/S)\rightarrow  \Bil^{s,\ev}(\Lambda(T_{\scr D(G)})\vert \Lambda(T_{G^{\ss}}))^{\scr W_G}\to 0,
$$
where $\Bil^{s,\ev}(\Lambda(T_{\scr D(G)})\vert \Lambda(T_{G^{\ss}}))^{\scr W_G}$ is the lattice of $\scr W_G$-invariant symmetric even bilinear form on $\Lambda(T_{\scr D(G)})$ that are integral on $\Lambda(T_{\scr D(G)})\otimes \Lambda(T_{G^{\ss}})$, where $G^{\ss}$ is the semisimplification of $G$.  
\end{theoremalpha}

For a family of curves $C/S$ of positive genus such that  $\End(J_{C_{\ov \eta}}) =\bbZ$, the map $\Pic(C)\rightarrow \Pic_{C/S}(S)$ is surjective  and satisfying one of the two assumptions of Theorem \ref{Main:PicG>0}, we give two alternative presentations of the relative Picard group of $\Bun_G^{\delta}(C/S)$ in Theorem \ref{T:altpres}.

The case of families of genus zero curves is easier and it is treated in detail in Theorem \ref{T:PicG=0} and Corollary \ref{C:PicG=0}. In particular, we will show that Theorem \ref{Main:PicG>0} holds true for families of genus zero curves under the assumption that $\scr D(G)$ is simply connected (see Theorem \ref{T:PicBunGAss}).

For a semisimple group $G$, the above Theorem gives the following 

\begin{coralpha}\label{MCor:PicG>0}
Let $C\to S$ be a family of curves of genus $g>0$ on a regular and integral quotient stack. Let $G$ be a
semisimple group and fix $\delta\in \pi_1(G)$.

 If one of the following two hypothesis is satisfied:
\begin{enumerate}[(a)]
\item   $G$ is simply connected;
\item  $\End(J_{C_{\ov \eta}}) =\bbZ$ for a geometric generic point $\ov \eta\to S$, the natural map $\Pic(C)\rightarrow \Pic_{C/S}(S)$ is surjective and $\RPic^0(C/S)$ is torsion-free;
\end{enumerate}
then the trasgression map 
\begin{equation*}
\tau_G^{\delta}=\tau^{\delta}_G(C/S):(\Sym^2 \Lambda^*(T_G))^{\mathscr W_G}\xrightarrow{\cong} \Pic \Bun_G^{\delta}(C/S),
\end{equation*}
is an isomorphism, functorial in $S$ and $G$.
\end{coralpha}
Note that, for $G$ semisimple, we have that $(\Sym^2 \Lambda^*(T_G))^{\mathscr W_G}\cong \bbZ^s$, where $s$ is the number of almost simple factors of the universal cover $G^{\sc}$ of $G$ (see Lemma \ref{L:coker-sym}).
In particular, if $G$ is simply connected and almost simple (i.e. $s=1$), then the above Corollary recovers (in the case where $C/S$ has a section) the isomorphism 
$ \Pic \Bun_G^{\delta}(C/S)\cong \bbZ$ established by Faltings \cite{Fa03}.

The above Theorem \ref{Main:PicG>0} (and hence Corollary \ref{MCor:PicG>0}) was known:
\begin{itemize}
\item for $G$ semisimple and simply connected and a family $C/S$ admitting a section, by Faltings \cite{Fa03};
\item for a curve $C$ over $k=\ov k$ if $\scr D(G)$ is simply connected, by Biswas-Hoffmann \cite{BH10}; 
\item for the universal family $\cC_{g,n}\to \cM_{g,n}$ (which satisfies the assumption \eqref{Ass-b} by Remarks \ref{R:eqRPic} and \ref{R:NSjump}\eqref{R:NSjump2}, and the Franchetta's conjecture)  by the work of the authors \cite[Thm. C(2)]{FV1} and \cite[Cor. 3.4]{FV2}, generalizing the work of the first author for  $G=GL_r$, $g\geq 2$ and $n=0$ in \cite[Thm. A]{Fri18}.
\end{itemize}
Note  that   Theorem \ref{Main:PicG>0} does not (necessarily) hold for a curve $C$ over $k=\ov k$ (which clearly does not satisfies assumption \eqref{Ass-b}) if the group $\scr D(G)$ is not simply connected, as it follows from the work of Biswas-Hoffmann \cite{BH10}.

In the final Section \ref{S:rigid} of the paper, we consider the rigidification 
\begin{equation*}\label{E:rigidINT}
\nu_G^{\delta}:\Bun_{G}^{\delta}(C/S) \to \Bun_{G}^{\delta}(C/S)\fatslash \scr Z(G):=\Bunr_{G}^{\delta}(C/S),
\end{equation*}
of the stack $\Bun_{G}^{\delta}(C/S)$ by the center $\scr Z(G)$ of $G$, which acts functorially on any $G$-bundle.

The Leray spectral sequence for the cohomology of the group $\Gm$ with respect to the $\scr Z(G)$-gerbe $\nu_G^{\delta}$ gives the  exact sequence  
\begin{equation*}
0\to   \Pic(\Bunr_G^{\delta}(C/S))\stackrel{(\nu_G^{\delta})^*}{\rightarrow} \Pic(\Bun_G^{\delta}(C/S)) \xrightarrow{\w_G^{\delta}} \Lambda^*(\scr Z(G)):=\Hom(\scr Z(G),\Gm) 
\end{equation*}
where the weight homomorphism $\w_G^{\delta}$ is given by the restricting the line bundles on $\Bun_G^{\delta}(C/S)$ to the fibers of $\nu_G^{\delta}$ and its cokernel measures the triviality of the $\scr Z(G)$-gerbe $\nu_G^{\delta}$ (see Section \ref{S:rigid} for more details).

In the next Theorem, we describe the Picard group of the rigidification $\Bunr_G^{\delta}(C/S)$ in terms of the Neron-Severi group $\NS(\Bunr_G^{\delta})$ (see Definition \ref{D:evGsc}) and the cokernel of the weight homomorphism $\w_G^{\delta}$ in terms of the cokernel of the evaluation homomorphism  $\ev_{\scr D(G)}^{\delta}$ (see Definition \ref{D:evGsc}).

\begin{theoremalpha}\label{Main:Rig}(see Theorems \ref{T:PicRig>0} and \ref{T:cokerwt})
Let $C\to S$ be a family of curves of genus $g>0$ on a regular and integral quotient stack. Assume that both the following conditions hold:
\begin{enumerate}[(i)]
\item $\End(J_{C_{\ov \eta}}) =\bbZ$ for a geometric generic point $\ov \eta\to S$and  the natural map $\Pic(C)\rightarrow \Pic_{C/S}(S)$ is surjective;
\item either $\RPic^0(C/S)$ is torsion-free or $\scr D(G)$ is simply connected.
\end{enumerate}

\begin{enumerate}
\item  \label{Main:Rig1} We have an exact sequence, functorial in $S$ and in $G$:
\begin{equation*}
0 \to \Lambda^*(G^{\ab})\otimes \RPic^0(C/S)  \xrightarrow{\ov{j_G^{\delta}}} \RPic(\Bunr_G^{\delta}(C/S))\xrightarrow{\ov{\gamma_G^{\delta}}} \NS(\Bunr_G^{\delta}),
\end{equation*}
such that 
\begin{equation*}
 \Im(\ov{\gamma_G^{\delta}})=
  \left\{b\in \NS(\Bunr_G^{\delta})\: : \: 
 \begin{aligned}
 & b(\delta\otimes x)+(g-1)b(x\otimes x)   \text{ is a multiple of } \delta(C/S) \\
&  \text{ for any } x\in \Lambda(T_G)
 \end{aligned}
 \right\}.
  \end{equation*}
  
  \item \label{Main:Rig2} We have an exact sequence, functorial in $S$ and in $G$:
 \begin{equation*}
 0\to \coker(\ov{\gamma_G^{\delta}})\xrightarrow{\partial_G^{\delta}} \Hom\left(\Lambda(G^{\ab}), \frac{\bbZ}{\delta(C/S)\bbZ}\right)\xrightarrow{\wt{\Lambda_{\ab}^*}} \coker(\w_G^{\delta})\xrightarrow{ \wt{\Lambda_{\scr D}^*}}\coker(\ev_{\scr D(G)}^{\delta})\to 0,
  \end{equation*}
 where $\delta(C/S)$ is the minimum  relative degree of a relative ample line bundle on $C/S$ (and zero if such a line bundle does not exist).
\end{enumerate}
\end{theoremalpha}
The Neron-Severi group $\NS(\Bunr_G^{\delta})$ and the evaluation homomorphism  $\ev_{\scr D(G)}^{\delta}$ are defined in Definition \ref{D:evGsc}.
We make explicit Theorem \ref{Main:Rig}\eqref{Main:Rig2} in the special cases $\delta(C/S)=1$ or $G=T$ a torus in Corollary \ref{C:PicRig>0}. 
 
The case of families of genus zero curves is easier and it is treated  in Theorem \ref{T:PicRig0}.

Theorem \ref{Main:Rig}\eqref{Main:Rig1}  was known:
\begin{itemize}
\item for the universal family $\cC_{g,n}\to \cM_{g,n}$ (which satisfies the assumptions  by Remarks \ref{R:eqRPic} and \ref{R:NSjump}\eqref{R:NSjump2}, and the Franchetta's conjecture)
 by the work of the authors \cite[Thm. 1.3]{FV2}, generalizing  \cite[Thm. B(i) and 1.5]{MV} (see also \cite{Kou91}) for $G=\Gm$, $n=0$ and  \cite[Thm. B(i) and Thm. A.2]{Fri18} (see also \cite{Kou93}) if $G=\GL_r$ and $n=0$. 
\end{itemize}

Theorem \ref{Main:Rig}\eqref{Main:Rig2}  was known:
\begin{itemize}
\item for a curve $C$ over $k=\ov k$ if $\scr D(G)$ is simply connected, by Biswas-Hoffmann \cite{BH12} (where also the case of an  arbitrary reductive group is treated);
\item for $G=\Gm$ and an arbitrary family $C/S$ in \cite{MR85};
\item for the universal family $\cC_{g,n}\to \cM_{g,n}$ by the work of the authors \cite[Thm. 1.4]{FV2}, generalizing \cite[Thm. 6.4]{MV} for $G=\Gm$, $n=0$, \cite[Cor. 3.3.2(i)]{Fri18}  for $G=\GL_r$, $n=0$,  \cite[Thm. B(i)]{FPbr} for $G=\GL_r$.
\end{itemize}

The computation of the cokernel of the weight homomorphism $\w_G^{\delta}$ (which coincides with the kernel of the obstruction homomorphism, see \eqref{E:seqLeray}) carried over in Theorem \ref{Main:Rig}\eqref{Main:Rig2}, will be used in our upcoming work \cite{FV4}, where we will compute the (cohomological) Brauer groups of $\Bun_{G}^{\delta}(C/S)$ and $\Bunr_{G}^{\delta}(C/S)$ for certain reductive groups, extending the work of Biswas-Holla \cite{BHol} for a fixed curve and a complex semisimple group, and the work of Fringuelli-Pirisi 
 \cite{FPbr} for the universal family $\cC_{g,n}/\cM_{g,n}$ and $G=\GL_r$.

\subsection*{Notations}

\begin{notations}\label{N:qs}
We denote by $k=\ov k$ an algebraically closed field of arbitrary characteristic $p\geq 0$. An algebraic stack will always be assumed locally of finite type over $k$ and quasi-separated, i.e. with quasi-compact and separated diagonal. In particular, all the schemes and algebraic spaces must be quasi-separated. Furthermore, any algebraic space is generically a scheme (see \cite[\href{https://stacks.math.columbia.edu/tag/06NH}{Proposition 06NH}]{stacks-project})
\end{notations}

\section{Preliminaries}\label{S:prel}

In this section, we collect some definitions and preliminaries results that we will need throughout the paper.

\subsection{Family of curves}\label{s:famcurves}

A \emph{curve} is a connected, smooth and projective scheme of dimension one over an algebraically closed field.  The genus of a curve $C$ is $g(C):=\dim H^0(C,\omega_C)$.

\vspace{0.1cm}

A \emph{family of curves} $\pi: C\to S$ is a proper, flat and representable morphism of algebraic stacks whose geometric fibers are curves. If all the geometric fibers of $\pi$ have the same genus $g$, then we say that $\pi:C\to S$ is a family of curves of genus $g$ (or a family of curves with relative genus $g$) and we set $g(C/S):=g$. Note that any family of curves $\pi:C\to S$ with $S$ connected is a family of genus $g$ curves for some $g\geq 0$.
For technical reasons, we will often restrict to families of curves  $\pi: C\to S$ over a regular and integral quotient stack $S$ over $k$. 

\vspace{0.1cm}

Given $g,n\geq 0$, we denote by $\cM_{g,n}$ the algebraic stack parametrizing families of curves of genus $g$ with $n$ ordered and pairwise disjoint sections, and by $\cC_{g,n}\to \cM_{g,n}$ the \emph{universal family of $n$-marked curves of genus $g$}. We omit the $n$ from the notation when it is zero. Note that $\cM_{g,n}$ is a DM (=Deligne-Mumford) stack if and only if $2g-2+n>0$.  

\vspace{0.1cm}

A family of curves $\pi: C\to S$ is \emph{locally projective} if, Zariski-locally on the base, it is a closed subscheme of $\bbP^n_S$ for some $n\geq 1$. Equivalently, $\pi:C\to S$ is locally projective if there exists 
a $\pi$-relatively ample line bundle on $C$, or equivalently a line bundle on $C$ having $\pi$-relative positive degree.
Note that a family of curves $\pi:C\to S$ is locally projective if either $g(C/S)\neq 1$  (because the relative dualizing line bundle $\omega_{\pi}$ has relative degree $2g(C/S)-2$)  or $\pi$ has a section $\sigma:S\to C$.
However, there are examples of families of genus one curves without sections that are not locally projective, e.g. the universal family $\cC_{1}\to \cM_{1}$  of curves of genus one (see Remark \ref{R:deltaCS}\eqref{R:deltaCS2}, and also \cite[XIII, 3.2]{Ray} and \cite{Zom} for other examples).
 
 \vspace{0.1cm}

In the results of this paper, an important role is played by the following invariant of the family $C/S$:

\begin{defin}\label{D:deltaCS}
For a family of curves $C/S$, we define 
$$
\delta(C/S):=\begin{cases}
	\min\{\deg_{\pi}(M)> 0\: : M\in \Pic(C) \}, &\text{ if } \pi \text{ is locally projective;} \\\
	0, &\text{ otherwise.}
\end{cases}
$$	
Equivalently,  if $\pi$ is locally projective, then $\delta(C/S)$ is the non-negative generator of the  non-zero image of the  relative degree map
\begin{equation}\label{E:rel-deg}
\begin{aligned}
\deg_{\pi}:\Pic(C)& \longrightarrow \bbZ,\\
M & \mapsto \deg_{\pi}(M).
\end{aligned}
\end{equation}
\end{defin}

\begin{rmk}\label{R:deltaCS}
\noindent 
\begin{enumerate}
\item \label{R:deltaCS1} Since the relative dualizing line bundle $\omega_{\pi}$ of $\pi:C\to S$ has relative degree equal to $2g(C/S)-2$, then 
$$
\delta(C/S) \: \text{ divides }\:  2g(C/S)-2.
$$
\item \label{R:deltaCS2} It follows from the weak Franchetta conjecture (see \cite{AC87}, \cite{Sch03}, \cite{FV3}) that, for the universal family $\cC_{g,n} \to \cM_{g,n}$ of $n$-marked curves of genus $g$, we have that 
$$
\delta(\cC_{g,n}/\cM_{g,n})=
\begin{cases}
|2g-2| & \text{ if }n=0,\\
1 & \text{ if } n>0. 
\end{cases}
$$
Notice that $\delta(\cC_1/\cM_1)=0$, which implies that  $\cC_1\to \cM_1$ is not locally projective. 
\item \label{R:deltaCS3}  The invariant $\delta(C/S)$ has been computed for many interesting families of curves:

\begin{itemize} 
\item If $C/S$ is the universal family of (smooth) complete intersections curves in $\bbP^r$ of type $(d_1,\ldots, d_{r-1})$, then $\delta(C/S)=d_1\ldots d_{r-1}$ (see \cite[Thm. 4.4]{Cil86}).
In particular, if  $C/S$ is the universal family of (smooth) plane curves of degree $d$, then $\delta(C/S)=d$ (see \cite[Thm. 4.2]{Cil86}).
\item If $C/S$ is the universal family of hyperplanes sections of the universal family of polarized genus $g\geq 3$ K3 surfaces, then $\delta(C/S)=2g-2$  (see \cite[Prop. 4.5]{Cil86}).
\item If $C/S$ is the universal family of one of the following 
\begin{itemize}
\item the unique component with general moduli of the Hilbert scheme of curves in $\bbP^r$ with $r\geq 3$ of genus $g$ and degree $d$;
\item the Severi variety of curves of genus $g$ together with a birational embedding as a nodal plane curve of degree $d$;
\item the Hurwtiz scheme of degree $d$ covers of $\bbP^1$ of genus $g$;
\end{itemize}
 and we assume that $\rho:=g-(r+1)(g+r-d)\geq 2$, then $\delta(C/S)=\gcd\{2g-2,d\}$ (see \cite[Thm. 1.1]{Cil87} and \cite[Thm. 5.3]{Cil86}). 
\item If $C/S$ is the universal family of hyperelliptic genus $g$ curves then (see \cite[Thm. 3.5]{MR85})
$$
\delta(C/S)=
\begin{cases}
4 & \text{ if }g \: \text{ is odd, }\\
2 & \: \text{ if }g \: \text{ is even.}
\end{cases}
$$
\end{itemize}
\item \label{R:deltaCS4} If $C/S$ is a family of curves of genus $g=0$, then 
$$
\delta(C/S)=
\begin{cases} 
1 & \text{ if }C/S \text{ is Zariski-locally trivial, }\\
2 & \text{ if }C/S \text{ is not Zariski-locally trivial.}\\
\end{cases}
$$ 
Note also that $C/S$ is Zariski locally trivial if and only if $C/S$ admits a section. 
\end{enumerate}
\end{rmk}

We will denote by $\Pic_{C/S}$ the \emph{relative Picard functor} of the family $C\to S$, which is an algebraic stack representable and locally of finite presentation over $S$
(by Artin's representability theorem, see e.g. \cite[\S 8.3]{BLR}). We have an exact sequence (see \cite[\S 8.1, Prop. 4]{BLR})
$$
0\to \Pic(S)\xrightarrow{\pi^*} \Pic(C) \to \Pic_{C/S}(S)\to \Br(S)\xrightarrow{\pi^*} \Br(C),
$$
where $\Br(-)=H^2(-,\Gm)$ denote the cohomological Brauer group.  In particular, we get an injective homomorphism
\begin{equation}\label{E:RPic-stack}
\RPic(C/S):=\frac{\Pic(C)}{\Pic(S)} \hookrightarrow \Pic_{C/S}(S).
\end{equation}
Several of our results are restricted to families such that the inclusion \eqref{E:RPic-stack} is an equality, or equivalently such that the natural map $\Pic(C)\to \Pic_{C/S}(S)$ is surjective. 
In the next Remark, we describe several families where this assumption holds true. 

\begin{rmk}\label{R:eqRPic}
The inclusion \eqref{E:RPic-stack} is  an equality in the following cases:
\begin{enumerate}
\item if $\pi:C\to S$ admits a section (because $\pi^*:\Br(S)\to \Br(C)$ is injective); in particular, if $C/S=\cC_{g,n}/\cM_{g,n}$ with $n>0$. Therefore, equality in \eqref{E:RPic-stack} can always be achieved after an \'etale cover of the base $S$. 
\item if $\Br(S)=0$, e.g. if $S$ is the spectrum of a field $k=\ov k$, or the spectrum of strictly henselian valuation ring, or a smooth curve over $k=\ov k$. 
\item  if $C/S=\cC_{g}/\cM_{g}$ and $g\geq 2$ by the strong Franchetta's conjecture (see \cite{Mes}).
\item if the relative Jacobian $J^d_{C/S}:=\Pic^d_{C/S}\to S$ in degree $d$ admits a Poincar\'e line bundle for every $d\in\mathbb Z$.
\item if $\delta(C/S)=1$ (see \cite[Lemma 2.2]{MR85}).
\end{enumerate}
\end{rmk}



Several of our results are restricted to families of curves $\pi: C\to S$ over an integral base  $S$, that satisfy the following condition 
\begin{equation}\label{E:End}
\End(J_{C_{\ov \eta}})=\bbZ,
 \end{equation}
where $\ov \eta$ is a geometric generic point of $S$. 

In the next result, we collect some sufficient and necessary conditions for \eqref{E:End} to hold.

\begin{prop}\label{P:EndZ}
Let $\pi: C\to S$ a family of curves over an integral base $S$ (locally of finite type over $k=\ov k$) and let $\ov \eta\to S$ be a geometric generic point. 
\begin{enumerate}
\item \label{P:EndZ1} If $\End(J_{C_{s}})=\bbZ$ for some $s\in S(k)$ then $\End(J_{C_{\ov \eta}})=\bbZ$.
\item \label{P:EndZ2} If $k$ is uncountable and $\End(J_{C_{\ov \eta}})=\bbZ$, then  $\End(J_{C_{s}})=\bbZ$ for $s$ very general in $S(k)$.
\item \label{P:EndZ3} If $k\not\subseteq  \ov{\mathbb{F}_p}$ then $\End(J_{C_{\ov \eta}})=\bbZ$ if and only if there exists $s\in S(\ov k)$ such that $\End(J_{C_{s}})=\bbZ$. 
\end{enumerate}
\end{prop}
Note that part \eqref{P:EndZ3} is false if $k=\ov k=\ov{\mathbb{F}_p}$, see \cite[Rmk. 1.12]{MP}.
\begin{proof}
Parts \eqref{P:EndZ1} and \eqref{P:EndZ2} follow from the well-know fact  (see the proof of \cite[Prop. 1.13]{MP}) that the jumping locus 
$$S(k)_{\rm{jump}}=\{s\in S(k)\: :\: \End(J_{C_{\ov \eta}})\hookrightarrow \End(J_{C_s}) \text{ is an isomorphism}\}$$
is the set of $k$-points of a union of countably many proper substacks of $S$.

Part \eqref{P:EndZ3} follows from the non-trivial result (proved in \cite{MP}, \cite{Chr}, \cite{Amb}) that $S(k)_{\rm{jump}}\neq S(k)$ if $k\neq \ov{\mathbb{F}_p}$.
\end{proof}

\begin{rmk}\label{R:NSjump}
\noindent 
\begin{enumerate} 
\item \label{R:NSjump2} The condition $\End(J_{C_{\ov \eta}})=\bbZ$ is known to hold for several general families of curves, e.g.:
\begin{itemize}
\item the  universal family $\cC_{g,n}\to \cM_{g,n}$ over an arbitrary field $k=\ov k$ by  \cite{Koi};
\item the universal family of hyperelliptic curves over an arbitrary field $k=\ov k$ by \cite[Thm. 6.5]{Mo76};
\item the  family of smooth curves belonging to a linear system on a regular smooth projective surface if $k=\bbC$ by \cite{Zar}.

\end{itemize}

\item \label{R:NSjump1}  The condition $\End(J_{C_{\ov \eta}})=\bbZ$ implies that the N\'eron-Severi of $J_{C_{\ov \eta}}$ is equal to 
\begin{equation}\label{E:rho=1}
\NS(J_{C_{\ov \eta}})=\bbZ\cdot [\Theta_{C_{\ov \eta}}],
\end{equation}
where $[\Theta_C]$ is the class of a theta divisor $J_{C_{\ov \eta}}$ (using the fact that $\NS(J_{C_{\ov \eta}})$ is the subgroup of $\End(J_{C_{\ov \eta}})$ consisting of elements that are invariant under the Rosari involution). But the viceversa is false, e.g. for an elliptic curve with complex multiplication. 

 Pirola has shown in \cite{Piro} that condition \eqref{E:rho=1} holds true for any family $C\to S$ of genus $g$ curves over $k=\bbC$ such that the image of $S$ in $\cM_g$ via the modular map has codimension less than or equal to $g-2$  (for example, the universal families of $k$-gonal curves of genus $g\geq 2$, for any $k\geq 2$). It would be interesting to known if for such families also the stronger statement $\End(J_{C_{\ov \eta}})=\bbZ$ holds true. 

\end{enumerate}
\end{rmk}

\subsection{Reductive groups}\label{s:redgrp}

A \emph{reductive group} (over $k=\ov k$) is a smooth and connected linear algebraic group over $k$ (i.e. a closed algebraic subgroup of $(\GL_n)_k$, or equivalently an affine group scheme of finite type over $k$), which does not contain non-trivial connected normal unipotent algebraic subgroups. To any reductive group $G$, we can associate a cross-like diagram of reductive groups 
\begin{equation}\label{E:cross}
\xymatrix{
& \scr D(G) \ar@{^{(}->}[d]^{\scr D} \ar@{->>}[dr] & \\
\scr R(G) \ar@{^{(}->}[r]_{\scr R}  \ar@{->>}[dr] & G \ar@{->>}[d]^{\ab}  \ar@{->>}[r]_{\ss}  & G^{\rm{ss}} \\
& G^{\ab} & \\
}
\end{equation}
where 
	\begin{itemize}
		\item   $\mathscr D(G):=[G,G]$ is the derived subgroup of $G$;
		\item   $G^{\mathrm{ab}}:=G/\mathscr D(G)$ is called the abelianization of $G$;
		\item   $\mathscr R(G)$ is the radical subgroup of $G$, which is equal (since $G$ is reductive) to the connected component  $\scr Z(G)^o$ of the center $\scr Z(G)$;
		\item   $G^{\rm{ss}}:=G/\mathscr R(G)$ is called the semisimplification of $G$.
	\end{itemize}		
In the above diagram, the horizontal and vertical lines are short exact sequences of reductive groups, the morphisms $\scr D(G)\twoheadrightarrow G^{\ss}$ and $\scr R(G)\twoheadrightarrow G^{\ab}$ are central isogenies of, respectively, semisimple groups and tori with the same kernel.
Since the two semisimple groups $\scr D(G)$ and $G^{\ss}$ are isogenous, they share the same simply connected cover, that we will denote by $G^{\sc}$, and the same adjoint quotient, that we will denote by $G^{\ad}$. 

\vspace{0.1cm}

Given a reductive group $G$ over $k$, we will denote by $T_G$ a maximal torus of $G$ and by $B_G$ the positive Borel subgroup of $G$ containing $T_G$. 
As usual, we denote by $\scr W_G:=N(T_G)/T_G$ its Weyl group. 

\vspace{0.1cm}

Given a torus $T$, we will denote by $\Lambda(T):=\Hom(\Gm,T)$ its cocharacter lattice and by  $\Lambda^*(T):=\Hom(T, \Gm)$ its character lattice. Note that these lattices are dual to each other via the bilinear map
$$
\begin{aligned}
\Lambda^*(T)\times \Lambda(T) & \longrightarrow \Hom(\Gm,\Gm)=\bbZ,\\
(\chi,d)& \mapsto \chi\circ d:=\chi(d). 
\end{aligned}
$$


\vspace{0.1cm}

The \emph{fundamental group} of $G$ is $\displaystyle \pi_1(G):=\frac{\Lambda(T_G)}{\Lambda_{\coroo}},$
where  $\Lambda_{\coroo}$ is the sublattice of $\Lambda(T_G)$ which is generated by integral linear combinations of coroots. The vertical exact sequence in \eqref{E:cross} induces 
an exact sequence of finitely generated abelian groups  
\begin{equation}\label{E:seq-pi1}
0\to \pi_1(\scr D(G))\to \pi_1(G) \xrightarrow{\pi_1(\ab)}\pi_1(G^{\ab})\to 0,
\end{equation}
which identifies $ \pi_1(\scr D(G))$ with the torsion subgroup of $\pi_1(G)$ and $\pi_1(G^{\ab})$ with the torsion-free quotient of $\pi_1(G)$. 

\vspace{0.1cm}




\subsection{The stack of $G$-bundles on a family of curves}\label{s:BunG}

In this subsection, $G$ will be a  connected and smooth linear algebraic group over $k=\ov k$.
Further restrictions on $G$, like  reductiveness, will be specified when needed.

For any family of curves $\pi: C\to S$, we denote by $\Bun_{G}(C/S)$ the  \emph{moduli stack of $G$-bundles on $C\to S$}, i.e. the algebraic stack over $S$ whose fiber over $V\to S$ is the groupoid of $G$-bundles $E\to C_V:=C\times_S V$. By definition, we have  a forgetful surjective morphism 
\begin{equation}\label{E:PhiG-rel}
\Phi_G=\Phi_{G}(C/S):\Bun_G(C/S) \longrightarrow S.
\end{equation}
The stack $\Bun_G(C/S)$ comes equipped with a universal $G$-torsor $\cP$ on the universal family $\wh \pi: \Bun_G(C/S)\times_S C\to \Bun_G(C/S)$. 

The stack $\Bun_G(C/S)$ is functorial in $S$, i.e. for any $V\to S$ we have that 
$$
\Bun_G(C_V/V)=\Bun_G(C/S)\times_S V.
$$
In particular, the fiber of $\Phi_G(C/S):\Phi_G(C/S)\to S$ over a $k$-point $s\in S$ is the moduli stack $\Bun_G(C_s)$ of $G$-bundles on the curve $C_s$. 

In the special case where the family of curves is the universal family $\cC_{g,n}\to \cM_{g,n}$ over the moduli stack of $n$-pointed curves of genus $g$, we set
$$
\Bun_{G,g,n}:=\Bun_G(\cC_{g,n}/\cM_{g,n}). 
$$
In particular, if the family $C\to S$ has constant relative genus $g=g(C/S)$ then we have that 
\begin{equation*}\label{E:reluniv2}
\Bun_{G}(C/S)=S\times_{\cM_{g}}\Bun_{G,g},
\end{equation*}
with respect to the modular morphism $S\to \cM_g$ associated to the family $C\to S$.


Any morphism of  connected and smooth linear algebraic groups $\phi:G\to H$ determines a morphism of stacks over $S$
\begin{equation}\label{E:fun1}
\begin{array}{lccc}
\phi_\sharp=\phi_{\sharp}(C/S):&\Bun_{G}(C/S)&\longrightarrow&\Bun_{H}(C/S)\\
&\Big(E\to C_V\Big)&\longmapsto &\Big(\phi_{\sharp}(E):=E\times^{\phi,G}H=(E\times H)/G\to C_V\Big)
\end{array}
\end{equation}
where the (right) action of $G$ on $E\times H$ is $(p,h).g:=(p.g,\phi(g)^{-1}h)$.

We collect in the following  Fact the geometric properties of   $\Bun_G(C/S)$ and of the forgetful morphism $\Phi_{G}(C/S)$.

\begin{fact}\label{F:propBunG}
Let $G$ be a connected and smooth linear algebraic group and let $\pi:C\to S$ be a family of curves. 
\begin{enumerate}
\item \label{F:propBunG1} The morphism $\Phi_{G}(C/S)$ is locally of finite presentation, smooth, with affine and finitely presented relative diagonal. 
\item \label{F:propBunG2} There is a functorial decomposition into a disjoint union
\begin{equation}\label{E:PhiGcomp}
\Phi_{G}(C/S): \coprod_{\delta\in \pi_1(G)} \Bun_{G}^{\delta}(C/S)\stackrel{\Phi_{G}^{\delta}(C/S)}{\longrightarrow} S,
\end{equation}
such that the morphism $\Phi_G^{\delta}(C/S): \Bun_{G}^{\delta}(C/S)\to S$ has geometric integral fibers for every $\delta\in \pi_1(G)$.


\item \label{F:propBunG6} If $G$ is reductive, then for any $\delta\in\pi_1(G)$ the morphism  $\Phi_G^{\delta}(C/S):\Bun_G^{\delta}(C/S)\to S$ is Stein, i.e. it is fpqc and cohomologically flat in degree zero (which means that the natural morphism $(\Phi_G^{\delta}(C/S)^{\sharp}:\cO_S \to  \Phi_G^{\delta}(C/S)_*(\cO_{\Bun_G^{\delta}(C/S)})$ is a universal isomorphism), and of relative dimension equal to  $(g-1)\dim G$.
\end{enumerate}
\end{fact} 
\begin{proof}
Part \eqref{F:propBunG1}: see \cite[Thm. 3.1]{FV1} and the references therein. Part \eqref{F:propBunG2}: see \cite[Thm. 3.1.1, Cor. 3.1.2]{FV1} and the references therein. 
Part \eqref{F:propBunG6}: see  \cite[Thm. 3.1.3]{FV1} and the proof of \cite[Prop. 3.3.2]{FV1}.
\end{proof}

Since the center $\scr Z(G)$ of a reductive group $G$ acts functorially on any $G$-bundle, we have that $\scr Z(G)$  sits functorially inside the automorphism group of any $T$-point $(C_T\to T, E)$ of $\Bun_{G}^{\delta}(C/S)$ for any $\delta\in \pi_1(G)$. Hence we can form the rigidification 
\begin{equation}\label{E:rigid}
\nu_G^{\delta}=\nu_G^{\delta}(C/S):\Bun_{G}^{\delta}(C/S)\to \Bun_{G}^{\delta}(C/S)\fatslash \scr Z(G):=\Bunr_{G}^{\delta}(C/S),
\end{equation}
which turns out to be a $\scr Z(G)$-gerbe, i.e. a gerbe banded by $\scr Z(G)$.

\section{The Picard group of $\Bun^{\delta}_G(C/S)$}

We will assume throughout this section that:
\begin{itemize}
\item[$\star$] $\pi:C\to S$ is a family of smooth curves of genus  $g=g(C/S)$ over an integral and regular quotient stack $S$ (locally of finite type over $k$). 
We will denote by $\ov \eta\to S$ a geometric generic point of $S$ and by $C_{\ov \eta}:=C\times_S \ov \eta$ the corresponding geometric generic fiber.
\end{itemize}
The aim of this section is to study the Picard group of $\Bun^{\delta}_G(C/S)$ for a reductive group $G$.

\subsection{Recollection on  $\Pic \Bun^{\delta}_G(C/S)$} 

In this subsection, we recall some facts from \cite{FV1, FV2} on the Picard group of $\Bun^{\delta}_G(C/S)$, for a reductive group $G$ and a fixed $\delta\in \pi_1(G)$.  The case $g=0$ is slightly different from the other cases and,  for unifying the statements, we need the following definition.

\begin{deflemma}\label{DL:gen-d}
Let $G$ be a reductive group and pick a maximal torus $\iota:T_G\hookrightarrow G$. 
\begin{enumerate}
\item We say that a cocharacter $d\in\Lambda(T_G)$ is \emph{generic} if its image $d^{\ss}\in\Lambda(T_{G^\ss})\subset \Lambda(T_{G^\ad})$, into the cocharacter lattice of the adjoint group, satisfies one of the following equivalent conditions (see \cite[Lemma 2.2.3]{FV1})
	\begin{enumerate}[(i)]
		\item The contraction homomorphism 
		$$
		\begin{aligned}
		 (d^{\ss},-): \left(\Sym^2 \Lambda^*(T_{G^{\sc}})\right)^{\mathscr W_G}=\Bil^{s,\ev}(\Lambda(T_{G^\sc}))^{\scr W_G}& \to \Lambda^*(T_{G^\sc})\\
		b & \mapsto b(d^\ss,-)
		\end{aligned}	
		$$
		is injective.
		\item Let $G^{\ad}=G_1^{\ad}\times \ldots \times G_s^{\ad}$ be the decomposition of $G^{\ad}$ into almost-simple factors (which are then automatically semisimple adjoint groups) and choose maximal tori in such a way that $T_{G^{\ad}}=T_{G_1^{\ad}}\times \ldots \times T_{G_s^{\ad}}$. Then $d^{ss}$ satisfies the following condition
		\begin{equation*}
			d^{ss}=d_1+\ldots+d_s\in \Lambda(T_{G^{\ad}})= \Lambda(T_{G_1^{\ad}})\oplus \ldots \oplus \Lambda(T_{G_s^{\ad}}) 
		\end{equation*}
with $d_i\neq 0$ for every $1\leq i \leq s$.	
	\end{enumerate}
\item 	Any $\delta\in \pi_1(G)$ admits a lift $d\in \pi_1(T_G)=\Lambda(T_G)$ which is generic. 
\end{enumerate}
\end{deflemma}

We collect in the following Fact the main results proved in \cite{FV1} for the Picard group of $\Bun_G^{\delta}(C/S)$. 

\begin{fact}\label{F:Pic-BunG}
Let 
$G$ be a reductive group and fix  $\delta\in\pi_1(G)$. 

\begin{enumerate}
\item \label{F:Pic-BunG0}
There exists an exact sequence 
$$
0\to \Pic(S) \to \Pic \Bun_G^{\delta}(C/S)\to \Pic \Bun_G^{\delta}(C_{\ov \eta})
$$
where the first map is the pull-back map and the second one is the restriction to the fiber of $C\to S$ over a geometric generic point $\ov \eta$ of $S$.

\item \label{F:Pic-BunG1} For any lift $d\in\pi_1(B_G)=\pi_1(T_G)=\Lambda(T_G)$ of $\delta$ (which is moreover generic if $g=0$), the pull-back maps along the inclusions $\iota:T_G\hookrightarrow B_G \hookrightarrow G$ 
$$\iota_\sharp^*:\Pic \Bun^{\delta}_G(C/S)\hookrightarrow \Pic \Bun^d_{B_G}(C/S)\xrightarrow{\cong} \Pic \Bun^d_{T_G}(C/S)$$
are such that  the first map is injective and the second one is an isomorphism. 

\item \label{F:Pic-BunG2}  There exists a unique homomorphism  (called \emph{transgression map})
\begin{equation}\label{E:trasgrG}
\tau_G^{\delta}=\tau^{\delta}_G(C/S):(\Sym^2 \Lambda^*(T_G))^{\mathscr W_G}\hookrightarrow \Pic \Bun_G^{\delta}(C/S),
\end{equation}
functorial in $S$ and $G$, and fitting in the following commutative diagram for any lift $d\in \pi_1(T_G)$ (which is generic if $g=0$) of $\delta\in \pi_1(G)$ 
 \begin{equation}\label{E:trasgrGT}
 \xymatrix{
(\Sym^2 \Lambda^*(T_G))^{\mathscr W_G}\ar[r]^{\tau^{\delta}_G} \ar@{^{(}->}[d] & \Pic \Bun_G^{\delta}(C/S)\ar@{^{(}->}[d]^{\iota_\sharp^*}, \\
\Sym^2 \Lambda^*(T_G)\ar[r]^(0.4){\tau^d_{T_G}} & \Pic \Bun_{T_G}^{d}(C/S), \\
\chi\cdot \chi' \ar[r] & \langle \cL_{\chi}, \cL_{\chi'}  \rangle_{\wh \pi} 
}
\end{equation}
where $\langle -,-\rangle_{\pi}$ is the Deligne pairing of the family $\wh \pi:C\times_S \Bun_{T_G}^d(C/S)\times\Bun_{T_G}^d(C/S) $ and $\cL_{\chi}$ is the pull-back of the universal $\Gm$-bundle (i.e. line bundle) over $C\times_S \Bun_{\Gm}(C/S)$ via the morphism $\id \times \chi_\#:C\times_S \Bun_G^{\delta}(C/S)\to C\times_S \Bun^{(\delta,\chi)}_{\Gm}(C/S)$.
\end{enumerate}
\end{fact}
\begin{proof}
Part \eqref{F:Pic-BunG0} follows from \cite[Prop. 5.1.1]{FV1}.  Part \eqref{F:Pic-BunG1} follows from \cite[Cor. 5.1.2, Thm. 6.0.1]{FV1}. Part \eqref{F:Pic-BunG2} follows from \cite[Thm. 5.0.1, Rmk. 5.2.2]{FV1} (whose proof works also for $g=0$). 
\end{proof}

Using Fact \ref{F:Pic-BunG}\eqref{F:Pic-BunG0}, we will often restrict to study the \emph{relative Picard group} 
\begin{equation}\label{E:RPic}
\RPic \Bun_G^\delta(C/S):=\frac{\Pic \Bun_G^{\delta}(C/S)}{\Pic(S)}.
\end{equation}

\subsection{The Picard group of $\Bun^d_T(C/S)$ for a torus $T$} 

The goal of this subsection is to determine the structure of the Picard group of $\Bun^{d}_T(C/S)$, where:
\begin{itemize}
\item[$\star$] $T$ is a torus and  $d\in \Lambda(T)$. 
\end{itemize}

We first introduce some natural line bundles on $\Bun_T^d(C/S)$.
For any character $\chi \in \Lambda^*(T)$ and $M\in \Pic(C)$, we consider the following line bundle on $C\times_S \Bun_T^d(C/S)$ 
$$\cL_{\chi}(M):=\cL_{\chi}\otimes p_1^*M,$$
where $p_1$ is the projection onto the first factor and $\cL_{\chi}$ is  the pull-back of the universal $\Gm$-bundle (i.e. line bundle) over $C\times_S \Bun_{\Gm}(C/S)$ via the morphism 
$\id\times \chi_\#:C\times_S \Bun_T^{d}(C/S)\to C\times_S \Bun^{(d,\chi)}_{\Gm}(C/S)$ induced by $\chi$. 

\begin{defin}\label{D:tautPic}
The \emph{tautological Picard group} of $\Bun_T^d(C/S)$ is the subgroup $\Pic^{\taut}\Bun_T^d(C/S)$  of  $\Pic\Bun_T^d(C/S)$ generated by the pull-back of the line bundles on $S$ and the following line bundles  (called \emph{tautological line bundles})
$$
\left\{d_{\wh \pi}(\cL_{\chi}(M))\: : \chi\in \Lambda^*(T), M\in \Pic(C)\right\}.
$$
where $d_{\wh \pi}$ denotes the determinant of cohomology with respect to the family $\wh \pi:C\times_S \Bun_T^d(C/S)\to \Bun_T^d(C/S)$. 

The \emph{relative tautological Picard group} of $\Bun_T^d(C/S)$ is the subgroup $\RPic^{\taut}\Bun_T^d(C/S)\subseteq \RPic \Bun_T^d(C/S)$ which is the image of  $\Pic^{\taut} \Bun_T^d(C/S)$ into the relative Picard group. 
\end{defin}

\begin{rmk}
The tautological Picard group $\Pic^{\taut}\Bun_T^d(C/S)$ contains all the line bundles 
$$\left\{\langle \cL_{\chi}(M),\cL_{\mu}(N)\rangle_{\wh \pi}: \chi,\mu\in \Lambda^*(T), M,N\in \Pic(C) \right\},$$
where $\langle - ,  - \rangle_{\wh \pi}$ is the Deligne pairing, since we have the following isomorphism
\begin{equation}\label{E:Del-det}
\langle \cL_{\chi}(M),\cL_{\mu}(N)\rangle_{\wh \pi}\cong d_{\wh \pi}(\cL_{\chi+\mu}(M\otimes N))\otimes d_{\wh \pi}(\cL_{\chi}(M))^{-1}\otimes d_{\wh \pi}(\cL_{\mu}(N))^{-1}\otimes d_{\wh \pi}(\cO).
\end{equation}
\end{rmk}

\vspace{0.1cm}

The Picard group of $\Bun_T^d(C/S)$ is endowed with a \emph{weight function} 
\begin{equation}\label{E:weight}	
	\Pic \Bun_T^d(C/S)\xrightarrow{w_T^{d}}\Lambda^*(T),
\end{equation}
which is defined as it follows. 
Let $\mathcal F$ be a  line bundle on $\Bun_T^d(C/S))$ and fix a $k$-point $\xi:=(C_s, E)$ of $\Bun_T^d(C/S)$ over $s\in S$. The automorphism group of $\xi$ acts on the fiber $\mathcal F_{\xi}$ of $\mathcal F$ over $\xi$. Since the torus $T$ is contained in the automorphism group of $\xi$, this defines an action of $T$ on $\mathcal F_{\xi}\cong k$ which  is given by a character of $T$.  This character, which is independent of the chosen $k$-point $\xi$ and on the chosen isomorphism $\mathcal F_{\xi}\cong k$,  coincides with $w_T^d(\mathcal F)$. Note that $w_T^d$ factors through the relative Picard group.

The weight function on the tautological Picard group is easily described. 

\begin{lem}
The weight function $w_T^d$ is given on tautological line bundles by 
\begin{equation}\label{E:wgt-det}
w_T^d(d_{\wh \pi}(\cL_{\chi}(M)))=[\chi(d)+\deg_\pi(M)+1-g]\chi.
\end{equation}
In particular, we have that 
\begin{equation}\label{E:wgt-Del}
w_T^d(\langle \cL_{\chi}(M),\cL_{\mu}(N)\rangle_{\wh \pi})=[\mu(d)+\deg_\pi(N)]\chi+[\chi(d)+\deg_\pi(M)]\mu. 
\end{equation}
\end{lem}
\begin{proof}
This is proved as in \cite[Prop. 4.1.2(1)]{FV1}.
\end{proof}

Using the weight function, we can now describe the structure of the Picard group of $\Bun_T^d(C/S)$ for a family of curves of genus $0$.

\begin{teo}\label{T:PicBunT-0}
Assume that  $g(C/S)=0$.    Then 
\begin{enumerate}
\item \label{T:PicBunT-01} $\Pic \Bun_T^d(C/S)=\Pic^{\taut}\Bun_T^d(C/S)$.
\item \label{T:PicBunT-02} The weight function 
$$
w_T^d:  \RPic \Bun_T^d(C/S)\hookrightarrow \Lambda^*(T)
$$
is injective on the relative Picard group. 
\item \label{T:PicBunT-03} The image of $w_T^d$ is equal to 
$$
\Im(w_T^d)=
\begin{cases}
\Lambda^*(T) & \text{ if } C/S \text{ is Zariski-locally trivial},\\
\{\chi \in \Lambda^*(T)\: : \chi(d) \: \text{ is even}\}  & \text{ if } C/S \text{ is not Zariski-locally trivial.}\
\end{cases}
$$
\end{enumerate}
\end{teo}
The above Theorem was proved for the universal family $\cC_{0,n}\to \cM_{0,n}$ (which is Zariski-locally trivial if and only if $n>0$, see Remark \ref{R:deltaCS}) in \cite[Thm. 4.2]{FV1}. 

\begin{proof}
Part \eqref{T:PicBunT-02}: since $g(C/S)=0$, the rigidification $\Bun_T^d(C/S)\fatslash T$ is isomorphic to $S$. 
Hence, the Leray spectral sequence for $\Gm$ associated to the $T$-gerbe  $\Bun_T^d(C/S)\to \Bun_T^d(C/S)\fatslash T$ gives the desired exact 
sequence
$$
0\to \Pic(S)=\Pic(\Bun_T^d(C/S)\fatslash T) \to \Pic \Bun_T^d(C/S)\xrightarrow{w_T^d} \Pic(BT)=\Lambda^*(T).
$$

Note that $\delta(C/S)=1$ or $2$ according to whether $C/S$ is Zariski-locally trivial or not (see Remark \ref{R:deltaCS}\eqref{R:deltaCS4}).
Hence, using the above exact sequence, in order to prove parts \eqref{T:PicBunT-01} and \eqref{T:PicBunT-03}, it is enough to prove the two equalities
\begin{equation}\label{E:eq-check}
w_T^d(\Pic^{\taut}\Bun_T^d(C/S))=\Im(w_T^d)=
\begin{cases}
\Lambda^*(T) & \text{ if } \delta(C/S)=1,\\
\{\chi \in \Lambda^*(T)\: : \chi(d) \: \text{ is even}\}  & \text{ if } \delta(C/S)=2.
\end{cases}
\end{equation}

We will distinguish two cases.

\un{Case 1:} $\delta(C/S)=1$.

In this case, the relative degree map \eqref{E:rel-deg} is an isomorphism. Hence, given any $\chi\in \Lambda^*(T)$, we can choose a line bundle $M_{\chi}\in \Pic(C)$ having relative degree equal to 
$g-\chi(d)$ and then formula \eqref{E:wgt-det} implies that 
$$
w_T^d(d_{\wh \pi}(\cL_{\chi}(M_{\chi})))=\chi.
$$
This shows that 
$$
w_T^d(\Pic^{\taut}\Bun_T^d(C/S))=\Lambda^*(T),
$$
which gives the desired two equalities in \eqref{E:eq-check}.

\un{Case 2:} $\delta(C/S)=2$.

In this case, the image of the relative degree map \eqref{E:rel-deg} is the subgroup $2\bbZ\subset \bbZ$.  We will divide the proof of the equalities in \eqref{E:eq-check} in three steps.

$\bullet$  $w_T^d(\Pic^{\taut}\Bun_T^d(C/S))\subseteq \{\chi \in \Lambda^*(T)\: : \chi(d) \: \text{ is even}\}$.

To prove this, since $\Pic^{\taut} \Bun_T^d(C/S)$ is generated the tautological line bundles $d_{\wh \pi}(\cL_{\chi}(M))$, it is enough to show, using formula \ref{E:wgt-det}, that 
$$
d_{\wh \pi}(\cL_{\chi}(M))(d)=[\chi(d)+\deg_{\pi}(M)+1]\chi(d)\:  \text{ is even,}  
$$
for any $\chi\in \Lambda^*(T)$ and any $M\in \Pic(C)$.  This is obvious if $\chi(d)$ is even. Otherwise, if $\chi(d)$ is odd, then it follows from the fact that
$\chi(d)+\deg_{\pi}(M)+1$ is even because $\deg_{\pi}(M)$ is even.

$\bullet$  $w_T^d(\Pic^{\taut}\Bun_T^d(C/S))\supseteq \{\chi \in \Lambda^*(T)\: : \chi(d) \: \text{ is even}\}$.

Let $\chi\in \Lambda^*(T)$ such that $\chi(d)$ is even. Then we can choice a line bundle $M_{\chi}\in \Pic(C)$ such that  $\deg_{\pi}(M_{\chi})=-\chi(d)$.  
We then have, using formula \ref{E:wgt-det}, that
$$
d_{\wh \pi}(\cL_{\chi}(M_{\chi}))=[\chi(d)+\deg_{\pi}(M_{\chi})+1]\chi=\chi,
$$
which shows that $\chi \in w_T^d(\Pic^{\taut}\Bun_T^d(C/S))$. 

$\bullet$ $\Im(w_T^d)\subseteq \{\chi \in \Lambda^*(T)\: : \chi(d) \: \text{ is even}\}$. 

Let $L$ be a line bundle on $\Bun_T^d(C/S)$ and set $\chi:=w_T^d(L)\in \Lambda^*(T)$. We want to show that $\chi(d)$ is even. 
Consider the following Cartesian diagram 
\begin{equation}\label{E:C*C}
\xymatrix{ 
\ar@{}[ddr]|\square &\Bun_T^d(C/S)\times_S C \ar[d]^{\wh \pi} \ar[dl] \ar@{}[dr]|\square & \Bun_T^d(C\times_S C /C)\times_C (C\times_S C) \ar@{}[ddr]|\square\ar[d]_{\wh{\pi_2}} \ar[l] \ar[dr] & \\
C \ar[dr]^{\pi}& \Bun_T^d(C/S)\ar[d]^{\Phi_S} \ar@{}[dr]|\square&  \Bun_T^d(C\times_S C /C) \ar[d]^{\Phi_C} \ar[l]^{\pi_T} \ar@/_1pc/[u]_{\wt \Delta}& C\times_S C\ar[dl]^{\pi_2}\\
& S & C\ar[l]^{\pi}\ar@/_2pc/[ur]_{\Delta} & 
}
\end{equation}
The pull-back $\pi_T^*(L)$ on $\Bun_T^d(C\times_S C/C)$ verifies $w_T^d(\pi_T^*(L))=\chi$, since the weight function is functorial with respect to base change. 
Consider now the line bundle $\langle \cL_{\chi}, \cL_0(\cO(\Delta))\rangle_{\wh{\pi_2}}=\langle\cL_{\chi}, \cO(\wt \Delta)\rangle_{\wh{\pi_2}}$ on $\Bun_T^d(C\times_S C/C)$. By \eqref{E:wgt-Del} and the fact that $ \cO(\Delta)$ has $\pi_2$-relative degree equal to $1$, we have that 
$$
w_T^d(\langle \cL_{\chi}, \cO(\wt \Delta)\rangle_{\wh{\pi_2}})=\chi.
$$
Therefore, since $\pi_T^*(L)$ and $\langle \cL_{\chi}, \cO(\wt \Delta)\rangle_{\wh{\pi_2}}$ have the same weight, part \eqref{T:PicBunT-02} implies that  
$$
\pi_T^*(L)=\langle \cL_{\chi}, \cO(\wt \Delta)\rangle_{\wh{\pi_2}}\otimes \Phi_C^*(M),
$$
for some $M\in \Pic(C)$. Taking the relative degree with respect to $\pi_T$, we obtain that
$$
0=\deg_{\pi_T}(\pi_T^*(L))=\deg_{\pi_T}(\langle \cL_{\chi}, \cO(\wt \Delta)\rangle_{\wh{\pi_2}})+\deg_{\pi_T}(\Phi_C^*(M))=
$$
$$
=\deg_{\pi_T}\wt \Delta^*(\cL_{\chi})+\deg_{\pi}(M)=\chi(d)+\deg_{\pi}(M),
$$
where we have used the isomorphism $\langle \cL_{\chi}, \cO(\wt \Delta)\rangle_{\wh{\pi_2}}\cong \wt \Delta^*(\cL_{\chi})$ has $\pi_T$-relative degree equal to $\chi(d)$. 
This implies that $\chi(d)$ is even since $\deg_{\pi}(M)$ is even because of the hypothesis that $\delta(C/S)=2$.
\end{proof}

We will now focus on families of curves of positive genus. In this case, the relative tautological Picard group is endowed with the following homomorphism. 

\begin{lem}\label{L:gammaT}
Assume that $g(C/S)>0$.  There exists a homomorphism
$$
\begin{aligned}
\gamma_T^d:\RPic^{\taut} \Bun_T^d(C/S)& \longrightarrow \Bil^s \Lambda(T)=(\Lambda^*(T)\otimes \Lambda^*(T))^s,\\
d_{\wh \pi}(\cL_{\chi}(M)) & \mapsto \chi\otimes \chi.
\end{aligned}
$$
In particular, we have that 
$$
\gamma_T^d(\langle \cL_{\chi}(M),\cL_{\mu}(N)\rangle_{\wh \pi})=\chi\otimes \mu+\mu\otimes \chi. 
$$
\end{lem}
\begin{proof}
This is proved as in \cite[Prop. 4.1.2(2)]{FV1} (see also \cite[Def./Lemma 3.5(2)]{FV2}).
\end{proof}

Using the above homomorphism $\gamma_T^d$, we can now describe the (relative) tautological Picard group of $\Bun_T^d(C/S)$ for families of curves of positive genus.

\begin{teo}\label{T:PicBunT>0}
Assume that  $g=g(C/S)>0$.   
\begin{enumerate}
\item  \label{T:PicBunT>01} Then there exists a commutative diagram with exact rows, functorial in $S$ and in $T$ 
\begin{equation}\label{E:3seq}
\xymatrix{
0 \ar[r] & \Lambda^*(T)\otimes \RPic^0(C/S) \ar[r]^{j_T^d} \ar@{^{(}->}[d]& \RPic^{\taut}\Bun_T^d(C/S)\ar@{=}[d] \ar[r]^{w_T^d\oplus \gamma_T^d}& \Lambda^*(T)\oplus \Bil^s(\Lambda(T))\ar@{->>}[d] &\\  
0 \ar[r] & \Lambda^*(T)\otimes \RPic(C/S) \ar[r]^{i_T^d} \ar@{^{(}->}[d]& \RPic^{\taut}\Bun_T^d(C/S)\ar@{=}[d] \ar[r]^{\gamma_T^d}& \Bil^s(\Lambda(T))\ar@{->>}[d]^{\epsilon} \ar[r] & 0\\  
0 \ar[r] & \Lambda^*(T)\otimes \RPic(C/S) \oplus \Sym^2 \Lambda^*(T) \ar[r]^(0,6){i_T^d\oplus \tau_T^d}& \RPic^{\taut}\Bun_T^d(C/S) \ar[r]^{\rho_T^d}& \Hom(\Lambda(T), \bbZ/2\bbZ) \ar[r] & 0\\  
}
\end{equation}
where 
\begin{itemize}
\item $\RPic(C/S)$ is the relative Picard group of $C/S$ and $\RPic^0(C/S)$ is the subgroup consisting of line bundles of relative degree $0$; 
\item $i_T^d$ is defined by 
$$i_T^d(\chi\otimes M)=\langle \cL_{\chi}, M\rangle_{\wh \pi};$$ 
\item $j_T^d$ is the restriction  of $i_T^d$ to $\Lambda^*(T)\otimes \RPic^0(C/S)$;
\item $\tau_T^d$ is the transgression map (see Fact \ref{F:Pic-BunG}\eqref{F:Pic-BunG2});
\item $\gamma_T^d$ is the map of Lemma \ref{L:gammaT};
\item $\rho_T^d$ is defined by 
$$
\rho_T^d(d_{\wh \pi}(\cL_{\chi}(M)))(x)=\chi(x)^2 \mod 2;
$$
\item $\epsilon$ is defined by 
$$
\epsilon(b)(x):=b(x,x) \mod 2.
$$
\end{itemize}
\item  \label{T:PicBunT>02}
The image of  $w_T^d\oplus \gamma_T^d$ is equal to 
$$
\left\{(\chi, b)\in  \Lambda^*(T)\oplus \Bil^s(\Lambda(T))\: : \: \delta(C/S)\vert \chi(x)-b(d,x)+(g-1)b(x,x) \text{ for every } x\in \Lambda(T)\right\}.
$$
\end{enumerate} 
\end{teo}

Most of the results of the previous Theorem were known for the universal family $\cC_{g,n}\to \cM_{g,n}$ with  $g\geq 1$: 
the exactness of the first line in \eqref{E:3seq} and part \eqref{T:PicBunT>02} were proved in \cite[Prop. 4.3.1]{FV1}; 
the exactness of the second line in \eqref{E:3seq} was proved  in \cite[Thm. 3.6]{FV2} (where it is also proved for an arbitrary reductive group); the exactness of the third line in \eqref{E:3seq} was proved for $g=1$  in \cite[Thm. B]{FV1} (while for genus $g\geq 2$ a different presentation is obtained in loc. cit.). 

\begin{proof}
Let us first prove part \eqref{T:PicBunT>01}. 
Observe that the diagram is commutative:  the commutativity of the two left squares and of the upper right square are obvious; the commutativity 
of the lower right square follows from the fact that $(\chi\otimes \chi)(x,y)=\chi(x)\chi(y)$. 

We now divide the rest of the proof of part \eqref{T:PicBunT>01} into three steps.

\un{Step 1:} The middle row is exact. 

First of all, observe that $i_T^d$ is well-defined since the Deligne pairing $\langle-,-\rangle_{\wh \pi}$ is bilinear and it is trivial on the pull-back of line bundle on $\Bun_T^d(C/S)$, and hence in particular on those coming from $S$. Consider now the following diagram 
$$
\xymatrix{
\Lambda^*(T)\otimes \RPic(C/S) \ar[r]^{i_T^d} \ar@{^{(}->}[d] & \RPic^{\taut}\Bun_T^d(C/S) \ar@{^{(}->}[d]\\
\Lambda^*(T)\otimes \Pic(C_{\ov \eta}) \ar[r]^{i_T^d(C_{\ov \eta})}& \RPic^{\taut}\Bun_T^d(C_{\ov \eta}) \\
}$$
where the vertical maps are restriction maps to the geometric generic fiber $C_{\ov \eta}$ of $C/S$ and the map $i_T^d(C_{\ov \eta})$ is the restriction of $i_T^d$ to the the geometric generic fiber.
Since the  restriction maps are injective by Fact \ref{F:Pic-BunG}\eqref{F:Pic-BunG0} and \cite[Prop. 2.3.2]{FV1}, and  the map $i_T^d(C_{\ov \eta})$ is injective by \cite[\S 3]{BH10}, we deduce that 
that also $i_T^d$ is injective. 

Lemma \ref{L:gammaT} implies that $\Im(i_T^d)\subseteq \ker(\gamma_T^d)$ and that $\gamma_T^d$ is surjective, using that $(\Lambda^*(T)\otimes \Lambda^*(T))^s$ is generated by the elements $\{\chi\otimes \chi\: : \chi\in \Lambda^*(T)\}$ (see \cite[\S 2.2]{FV1}). 

It remains to prove that $\ker(\gamma_T^d)\subseteq \Im(i_T^d)$. To this aim, we fix a basis $\{\chi_1,\ldots,\chi_r\}$ of $\Lambda^*(T)$. Then $\RPic^{\taut}(\Bun_T^d(C/S))$ is generated 
by the line bundles
$$
\{d_{\wh \pi}(\cL_{\chi_i}), \langle \cL_{\chi_j},\cL_{\chi_k}\rangle_{\wh \pi}, \langle \cL_{\chi_h}, M\rangle_{\wh \pi}\}
$$ 
with  $1\leq i, h\leq r$, $1\leq j\leq k\leq r$, $M\in \RPic(C/S)$. The line bundles $ \langle \cL_{\chi_h}, M\rangle_{\wh \pi}$ belong to $\Im(i_T^d)\subseteq \ker(\gamma_T^d)$ as proved above. 
Suppose that we have a line bundle on $\Bun_T^d(C/S)$
$$
K:=\bigotimes_i d_{\wh \pi}(\cL_{\chi_i})^{a_i}\bigotimes_{j\leq k} \langle \cL_{\chi_j},\cL_{\chi_k}\rangle_{\wh \pi}^{b_{jk}} \: \text{ for some } a_i, b_{jk}\in \bbZ
$$
whose class in the relative tautological Picard group belongs to $\ker(\gamma_T^d)$. Then Lemma \ref{L:gammaT} implies that 
\begin{equation}\label{E:gamma0}
0=\gamma_T^d\left(\bigotimes_i d_{\wh \pi}(\cL_{\chi_i})^{a_i}\bigotimes_{j\leq k} \langle \cL_{\chi_j},\cL_{\chi_k}\rangle_{\wh \pi}^{b_{jk}}\right)=\sum_i a_i \chi_i\otimes\chi_i+\sum_{j\leq k}
b_{jk}(\chi_j\otimes \chi_k+\chi_k\otimes \chi_j).
\end{equation}
Since a basis of $\Bil^s(\Lambda(T))=(\Lambda^*(T)\otimes \Lambda^*(T))^s$ is given by $\{\{\chi_i\otimes \chi_i\}_i\cup \{\chi_j\otimes \chi_k+\chi_k\otimes \chi_j\}_{j<k}\}$, the relation \eqref{E:gamma0}
is equivalent to 
$$
\begin{sis}
b_{jk}=0 & \text{ for any }j<k,\\
a_i+2b_{ii}=0 & \text{ for any }i. 
\end{sis}
$$
Hence the line bundle $K$ is equal to 
$$
K=\bigotimes_i [d_{\wh \pi}(\cL_{\chi_i})^{-2}\bigotimes \langle \cL_{\chi_i}, \cL_{\chi_i}\rangle_{\wh \pi}]^{b_{ii}}= \bigotimes_i [\langle \cL_{\chi_i}, \omega_{\wh \pi}\rangle_{\wh \pi}\otimes d_{\wh \pi}(\cO)^{-2}]^{b_{ii}}
$$
where the last equality follows from \cite[Rmk. 3.5.1]{FV1}. This implies that the class of $K$ in the relative Picard group belongs to $\Im(i_T^d)$ and we are done. 

\un{Step 2:} The first row is exact.

Indeed, the map $j_T^d$ is injective since it is the restriction of the injective map $i_T^d$. Using the exactness of the middle row, the kernel of $w_T^d\oplus \gamma_T^d$ is equal to 
$$
\ker(w_T^d\oplus \gamma_T^d)=\left\{i_T^d(\chi\otimes M)\: : \: \w_T^d(i_T^d(\chi\otimes M))=0\right\}.
$$
By the definition of $i_T^d$ and \eqref{E:wgt-Del}, we compute 
$$
\w_T^d(i_T^d(\chi\otimes M))=w_T^d(\langle \cL_{\chi}, M\rangle_{\wh \pi})=\deg_{\pi}(M)\chi.
$$
 Hence the kernel of $w_T^d\oplus \gamma_T^d$ coincides with the image via $i_T^d$ (or equivalently via $j_T^d$) of $\Lambda^*(T)\otimes \RPic^0(C/S)$, and we are done.

\un{Step 3:} The last row is exact.

Indeed, the composition $\gamma_T^d\circ \tau_T^d$ is given by 
$$
(\gamma_T^d\circ \tau_T^d)(\gamma\cdot \gamma')=\gamma_T^d(\langle \cL_{\gamma}, \cL_{\gamma'}\rangle_{\wh \pi})=\gamma\otimes \gamma'+\gamma'\otimes \gamma
$$
Hence, $\gamma_T^d\circ \tau_T^d$ is injective and its image is the subgroup $\Bil^{s,\ev}(\Lambda(T))\subset \Bil^s(\Lambda(T))$ of even symmetric bilinear forms (see \cite[\S 2.2]{FV1}). 
We now deduce the exactness of the last row from the exactness of the second row and the fact that $\Bil^{s,\ev}(\Lambda(T))$ is the kernel of the surjective map $\epsilon$. 

Let us prove now part \eqref{T:PicBunT>02}.  We first observe that the image of $w_T^d\oplus \gamma_T^d$ is contained in 
$$
I:=\left\{(\chi, b)\in  \Lambda^*(T)\oplus \Bil^s(\Lambda(T))\: : \: \delta(C/S)\vert \chi(x)-b(d,x)+(g-1)b(x,x) \text{ for every } x\in \Lambda(T)\right\}.
$$
Indeed, for the generators of $\RPic^{\taut}\Bun_T^d(C/S)$ we compute (using \eqref{E:wgt-det} and Lemma \ref{L:gammaT}) that 
$$
(w_T^d,\gamma_T^d)(d_{\wh \pi}(\cL_{\chi}(M)))=(\chi(d)+\deg_{\pi}(M)+1-g)\chi, \chi\otimes \chi).
$$
Hence for any $x\in \Lambda(T)$, we get 
$$
[\chi(d)+\deg_{\pi}(M)+1-g] \chi(x) -\chi(d)\chi(x)+(g-1)\chi(x)^2=[\deg_{\pi}(M)+1-g] \chi(x) +(g-1)\chi(x)^2=
$$
$$
\begin{aligned}
=\deg_{\pi}(M) \chi(x) +(g-1)[\chi(x)^2-\chi(x)]& \equiv \deg_{\pi}(M) \chi(x)  \mod 2g-2 \\
& \equiv 0 \mod \delta(C/S),
\end{aligned}
$$
where in the last congruence relation we used that $\deg_{\pi}(M) $ is a multiple of $\delta(C/S)$ (by definition of $\delta(C/S)$) and that $\delta(C/S)$ is a multiple of $2g-2$ (by Remark \ref{R:deltaCS}\eqref{R:deltaCS1}). 

Next, in order to show that the inclusion $\Im(w_T^d\oplus \gamma_T^d)\subseteq I$ is an equality, we consider the following commutative diagram with exact vertical columns
$$
\xymatrix{
\Im(w_T^d\oplus \gamma_T^d)\cap (\Lambda^*(T)\times \{0\}) \ar@{^{(}->}[r]  \ar@{^{(}->}[d]  & I \cap (\Lambda^*(T)\times \{0\})  \ar@{^{(}->}[r]  \ar@{^{(}->}[d]  & \Lambda^*(T)  \ar@{^{(}->}[d] \\
\Im(w_T^d\oplus \gamma_T^d) \ar@{^{(}->}[r] \ar@{->>}[d]& I \ar@{^{(}->}[r] \ar@{->>}[d] & \Lambda^*(T)\oplus \Bil^s(\Lambda(T)) \ar@{->>}[d]^{p_2}\\
p_2(\Im(w_T^d\oplus \gamma_T^d))   \ar@{^{(}->}[r] & p_2(I)  \ar@{^{(}->}[r] & \Bil^s(\Lambda(T))
}
$$
We now conclude using that 

$\bullet$ $p_2(\Im(w_T^d\oplus \gamma_T^d))=\Bil^s(\Lambda(T))$ as it follows from the right upper square of the diagram \eqref{E:3seq};

$\bullet$ the inclusion 
$$\Im(w_T^d\oplus \gamma_T^d)\cap (\Lambda^*(T)\times \{0\}) \subseteq I \cap (\Lambda^*(T)\times \{0\})=\delta(C/S)\cdot \Lambda^*(T)$$ 
is an equality since it coincides with the boundary homomorphism coming from the snake lemma applied to the first two horizontal lines of the diagram  \eqref{E:3seq}.
\end{proof}

After having obtained in Theorem \ref{T:PicBunT>0} a complete description of the tautological Picard group of $\Bun_T^d(C/S)$ for families of curves of positive genus, the next natural question 
is to determine when the entire Picard group coincides with the tautological subgroup. A sufficient condition is provided by the following 

\begin{teo}\label{T:all-taut}
Assume that  $g=g(C/S)>0$.   If $\End(J_{C_{\ov \eta}}) =\bbZ$ and $\RPic(C/S)=\Pic_{C/S}(S)$  then 
$$
\Pic \Bun_T^d(C/S)=\Pic^{\taut} \Bun_T^d(C/S).
$$
\end{teo}
The above Theorem was shown in \cite{BH10} (see also \cite[Prop. 4.1.1]{FV1}) for a curve $C$ over $S=\Spec k$, with $k$ an algebraically closed field,  in which case the second assumption is automatic and the first assumption  $\End(J_{C_{\ov \eta}}) =\bbZ$ is also necessary for the  Theorem to hold.  Theorem \ref{T:all-taut} recovers the equality $\Pic\Bun_T^d(\cC_{g,n}/\cM_{g,n})=\Pic^{\taut}\Bun_T^d(\cC_{g,n}/\cM_{g,n})$, which was shown in \cite[Thm. 4.1]{FV1} for the universal family $\cC_{g,n}\to \cM_{g,n}$ (with  $g\geq 1$),  since the assumption of the Theorem is known to hold for $\cC_{g,n}\to \cM_{g,n}$ (see  Remarks \ref{R:eqRPic} and \ref{R:NSjump}\eqref{R:NSjump2}).

\begin{proof}
We will distinguish three cases.

\un{Case I:} $d=0$ and $\pi: C\to S$ admits a section $\sigma$.

Since $d=0$, we have the canonical identifications 
 $$
 \begin{sis}
 & \Bun_T^0(C/S)= \cJ^0(C/S)\otimes \Lambda(T), \\
 &  \Bun_T^0(C/S)\fatslash T=J^0(C/S)\otimes \Lambda(T),\\
 \end{sis}
 $$
 where  $\cJ^0(C/S)$ is the Jacobian stack of  $C/S$ and $J^0(C/S)=\cJ^0(C/S)\fatslash T$ is the Jacobian scheme of $C/S$. 
The Theorem is implied in this case by the following commutative  diagram with exact rows 
\begin{equation}\label{E:ab-scheme}
\xymatrix{
0 \ar[r] & \Lambda^*(T)\otimes \RPic^0(C/S) \ar[r]^{j_T^0}\ar[d]^{\cong} &\RPic^{\taut}\Bun_T^0(C/S)\ar[r]^{w_T^0\oplus \gamma_T^0} \ar@{^{(}->}[d]& \Lambda^*(T)\oplus \Bil^s(\Lambda(T))\ar[r] 
\ar[d]^{\cong}& 0\\
0 \ar[r] & (J^0(C/S)\otimes \Lambda(T))^{\vee}(S) \ar[r] &\RPic\Bun_T^0(C/S)\ar[r]^{\ov{\res}_T^{0}(C_{\ov \eta})} & \NS(\Bun_T^0(C_{\ov \eta})) & \\
}
\end{equation}
where
\begin{itemize}
\item the first line is exact by Theorem \ref{T:PicBunT>0} using that $\delta(C/S)=1$ since $\pi: C\to S$ has a section by assumption;
\item the second exact line comes from the fact that the rigidification $\Bun_T^0(C/S)\fatslash T=J^0(C/S)\otimes \Lambda(T)\to S$ is an abelian scheme (see e.g. the proof of \cite[Prop. 3.6]{FP16}), where  $(J^0(C/S)\otimes \Lambda(T))^{\vee}=J^0(C/S)^{\vee} \otimes \Lambda^*(T)$ is the dual abelian scheme and the last map is the restriction homomorphism onto the geometric generic fiber;
\item the left vertical isomorphism follows from the fact that $\RPic^0(C/S)=J^0(C/S)^{\vee}(S)$;
\item the right vertical isomorphism follows from the fact that (see \cite[\S 3.2]{BH10})
$$
\NS(\Bun_T^0(C_{\ov \eta}))=\Lambda^*(T)\oplus \Hom^s(\Lambda(T)\otimes \Lambda(T), \End(J_{C_{\ov \eta}})),
$$
where $\Hom^s(\Lambda(T)\otimes \Lambda(T), \End(J_{C_{\ov \eta}}))$ is the set of bilinear forms $\Hom(\Lambda(T)\otimes \Lambda(T), \End(J_{C_{\ov \eta}}))$ that are symmetric with respect to the Rosati involution, together with the assumption that $\End(J_{C_{\ov \eta}}) =\bbZ$. 
\item the commutativity of the left square follows from the proof of \cite[Lemma 4.2.1]{FV1}\footnote{the proof of loc. cit. for the universal family $\cC_{g,n}/\cM_{g,n}$ extends mutatis mutandis to an arbitrary family $C/S$} together with formula \eqref{E:Del-det};
\item the commutativity of the right square follows from \cite[Prop. 4.1.2]{FV1}\footnote{the proof of loc. cit. for the universal family $\cC_{g,n}/\cM_{g,n}$ extends mutatis mutandis to an arbitrary family $C/S$}.
\end{itemize}

\un{Case II:} $\pi: C\to S$ admits a section $\sigma$.

Fix an isomorphism $T\cong \Gm^r$, which induces a canonical isomorphism $\Lambda(T)\cong \bbZ^r$ under which $d$ corresponds to $(d_1,\ldots, d_r)$. 
Then we have an isomorphism 
$$\Bun_T^d(C/S)\cong \Bun_{\Gm}^{d_1}(C/S)\times_S \ldots \times_S \Bun_{\Gm}^{d_r}(C/S).$$
Using the section $\sigma$, we define an isomorphism 
$$
\begin{aligned}
\mathfrak t: \Bun_{\Gm}^{d_1}(C/S)\times_S \ldots \times_S \Bun_{\Gm}^{d_r}(C/S) & \xrightarrow{\cong} \Bun_{\Gm}^{0}(C/S)\times_S \ldots \times_S \Bun_{\Gm}^{0}(C/S)\cong \Bun^0_T(C/S)\\
(L_1,\ldots, L_r) & \mapsto (L_1(-d_1\sigma), \ldots, L_r(-d_r\sigma))
\end{aligned}
$$
By construction, we have that $(\mathfrak t\times \id)^*(\cL_{e_i})=\cL_{e_i}(-d_i\sigma)$, where $\{e_i\}$ is the canonical basis of $\bbZ^r\cong \Lambda^*(T)$. 
Hence, using the functoriality of the determinant of cohomology, we deduce that the pull-back isomorphism 
$$\mathfrak t^*:\Pic \Bun^0_T(C/S)\xrightarrow{\cong} \Pic\Bun_T^d(C/S)$$ 
preserves the tautological subgroups.  Hence the conclusion is implied by Case I. 



\un{Case III:} General case.

Let $\pi: C\to S$ be an arbitrary family of curves (i.e. possibly without a section). Fix a smooth cover $S'\to S$ such that the family $C':=C\times_SS'\to S'$ admits a section (e.g. $C\to S$). By Case II, we have the following morphism of exact sequences
\begin{equation}\label{E:co-equalizer}
\xymatrix{
0\ar[r]&\Lambda^*(T)\otimes\RPic(C'/S')\ar[d]^*{\mathrm{Id}_{\Lambda^*(T)}\otimes(p_1^*-p^*_2)}\ar[r]&\RPic(\Bun^d_T(C'/S'))\ar[d]^*{p_1^*-p^*_2}\ar[r]&\Bil^s(\Lambda(T))\ar[d]^0\ar[r]&0\\
0\ar[r]&\Lambda^*(T)\otimes\RPic(C''/S'')\ar[r]&\RPic(\Bun^d_T(C''/S''))\ar[r]&\Bil^s(\Lambda(T))\ar[r]&0
}
\end{equation}
where $(C''\to S''):=(C'\times_SS'\to S'\times_SS')$ and $p_1,p_2:C''\to C'$ denotes the natural projections. 

Consider the following commutative diagram of abelian groups
$$
\xymatrix{
	0\ar[r]&\Lambda^*(T)\otimes\RPic(C/S)\ar@{=}[d]\ar[r]^{i_T^d}&\RPic^{\taut}(\Bun^d_T(C/S))\ar@{^{(}->}[d]^{\iota^{\taut}}\ar[r]&\Bil^s(\Lambda(T))\ar[d]\ar[r]&0\\
	0\ar[r]&\Lambda^*(T)\otimes\RPic(C/S)\ar@{^{(}->}[d]\ar[r]^{\iota^{\taut}\circ i_T^d}&\RPic(\Bun^d_T(C/S))\ar@{^{(}->}[d]\ar[r]&K:=\coker\{\iota^{\taut}\circ i_T^d\}\ar[d]\ar[r]&0\\
	0\ar[r]&\Lambda^*(T)\otimes \Pic_{C/S}(S)\ar[r]&\ker\{p_1^*-p_2^*\}\ar[r]&\Bil^s(\Lambda(T))
}
$$
where all the rows are exact: the first one by Theorem \ref{T:PicBunT>0}, the second one by definition and the third one is obtained by taking the kernels of the vertical arrows in Diagram \eqref{E:co-equalizer}. 

By applying the snake lemma at the first and the third row and using the assumption $\RPic(C/S)=\Pic_{C/S}(S)$, we deduce that the central vertical arrows in the middle must be equalities, and the  have the assertion.
\end{proof}

\begin{rmk}If we remove the hypothesis $\RPic(C/S)=\Pic_{C/S}(S)$, Theorem \ref{T:all-taut} is false. Indeed, from the proof of Case III, we may deduce the following exact sequence of abstract groups
		$$
		0\to \Lambda^*(T)\otimes \Pic_{C/S}(C)\to \Pic_{\Bun_T^d(C/S)/S}(S)\to \Bil^s(\Lambda(T))\to 0,
		$$ 
	where $\Pic_{\Bun_T^d(C/S)/S}$ denotes the relative Picard sheaf of the relative moduli stack $\Bun_T^d(C/S)\to S$. If we assume $\RPic(C/S)\subsetneq\Pic_{C/S}(S)$ and $\RPic(\Bun_T^d(C/S))= \Pic_{\Bun_T^d(C/S)/S}(S)$ (e.g. if $d=0$, because $\Bun_T^0(C/S)\to S$ has sections), we get $\RPic^{\taut}(\Bun_T^d(C/S))\subsetneq\RPic(\Bun_T^d(C/S))$.
\end{rmk}

\subsection{The Picard group of $\Bun^{\delta}_G(C/S)$ for $G$ reductive} 

In this subsection, we study when the Picard group of $\Bun^{\delta}_G(C/S)$ is generated by the image of the pull-back along the abelianization map
$$\ab_{\sharp}:\Bun^{\delta}_G(C/S)\to  \Bun^{\delta^{\ab}}_{G^{\ab}}(C/S),$$ 
where $\delta^{\ab}:=\pi_1(\ab)(\delta)$, and the image of  the transgression map \eqref{E:trasgrG}.

More precisely, we investigate when the following commutative diagram, which is functorial in $G$ and $S$,
\begin{equation}\label{E:push-diag}
\xymatrix{
 \Sym^2 \Lambda^*(G^{\ab}) \ar@{^{(}->}[r]^{\Sym^2 \Lambda_{\ab}^*} \ar[d]^{\tau_{G^{\ab}}^{\delta^{\ab}}}& \left(\Sym^2 \Lambda^*(T_{G})\right)^{\mathscr W_G} \ar[d]^{\tau_{G}^\delta}\\
   \Pic \Bun_{G_{\ab}}^{\delta^{\ab}}(C/S) \ar@{^{(}->}[r]^{\ab_{\sharp}^*} &  \Pic \Bun_G^{\delta}(C/S)
}
\end{equation}
is a push-out diagram, where $\tau_G^{\delta}$ and $\tau_{G^{\ab}}^{\delta^{\ab}}$ are, respectively, the transgression maps for $G$ and $G^{\ab}$ (see Fact \ref{F:Pic-BunG}\eqref{F:Pic-BunG2}),
$\Sym^2 \Lambda_{\ab}^*$ is the injective homomorphism induced by the inclusion $\Lambda_{\ab}^*:\Lambda^*(G^{\ab})\hookrightarrow \Lambda^*(T_{G})$ which is invariant under the $\scr W_G$-action on  $\Lambda^*(T_{G})$, and $\ab_{\sharp}^*$ is injective by the following 

\begin{lem}\label{L:inj-ab*}
The pull-back map 
$$
\ab_{\sharp}^*:  \Pic \Bun_{G_{\ab}}^{\delta^{\ab}}(C/S)\to \Pic \Bun_G^{\delta}(C/S)
$$
is injective. 
\end{lem}
\begin{proof}
The composition 
$$\phi:T_G\stackrel{\iota}{\hookrightarrow} G \stackrel{\ab}{\twoheadrightarrow} G^{\ab}$$
is a surjective homomorphism of tori, and hence it admits a section. This implies that, for any lift $d\in \pi_1(T_G)$ of $\delta\in \pi_1(G)$,  the homomorphism
$$
\phi_{\sharp}^*:\Pic \Bun_{G_{\ab}}^{\delta^{\ab}}(C/S)\xrightarrow{\ab_\sharp^*} \Pic \Bun_G^{\delta}(C/S) \xrightarrow{\iota_\sharp^*} \Pic \Bun_{T_G}^{d}(C/S)
$$
is injective. Hence, also the homomorphism $\ab_{\sharp}^*$ must be injective.
\end{proof}


We will use the following easy lemma on push-out diagrams of abelian groups. 

\begin{lem}\label{L:push}
Consider a commutative diagram of abelian groups
\begin{equation}\label{E:diag-Ab}
\xymatrix{
A \ar@{^{(}->}[r]^i \ar[d] & B\ar [d]\\
C \ar@{^{(}->}[r]^{j} & D
}
\end{equation}
where $i$ and $j$ are injective homomorphisms. Then \eqref{E:diag-Ab} is a push-out diagram if and only if   the natural map $\coker i\to \coker j$ is an isomorphism. 
\end{lem}
\begin{proof}
This is easy and left to the reader. 
\end{proof}

Hence, the diagram \eqref{E:push-diag} is a push-out diagram if and only if the induced homomorphism 
$$
\coker(\Sym^2 \Lambda_{\ab}^*)\rightarrow \coker(\ab_{\sharp}^*)
$$
is an isomorphism. 

The cokernel of $\Sym^2 \Lambda_{\ab}^*$ is described explicitly as follows. 

\begin{lem}\label{L:coker-sym}
\noindent 
\begin{enumerate}
\item \label{L:coker-sym1} We have an exact sequence of lattices, functorial in $G$,
\begin{equation}\label{E:seq-LTG}
\begin{aligned}
& 0\to \Sym^2 \Lambda^*(G^{\ab}) \xrightarrow{\Sym^2\Lambda_{\ab}^*}   \Sym^2\Lambda^*(T_G)^{\scr W_G}  \xrightarrow{\res_{\scr D}} \Bil^{s,\ev}(\Lambda(T_{\scr D(G)})\vert \Lambda(T_{G^{\ss}}))^{\scr W_G}\to 0,\\
 &  \hspace{6cm} b  \mapsto b_{|\Lambda(T_{\scr D(G)})\otimes \Lambda(T_{\scr D(G)})}  \\
\end{aligned}
\end{equation} 
where we identify $\Sym^2\Lambda^*(T_G)$ with the lattice $\Bil^{s,\ev}(\Lambda(T_G))$ of even symmetric  bilinear forms and the last group is defined by 
$$
\Bil^{s,\ev}(\Lambda(T_{\scr D(G)})\vert \Lambda(T_{G^{\ss}}))^{\scr W_G}:=\left\{b\in \Bil^{s,\ev}(\Lambda(T_{\scr D(G)}))^{\scr W_G}\: : \: b^{\bbQ}_{|\Lambda(T_{\scr D(G)})\otimes \Lambda(T_{G^{\ss}})}\: \text{ is integral,}\right\}
$$
where $b^{\bbQ}$ is the  extension of $b$ to a rational bilinear symmetric form on $\Lambda(T_{G^{\ss}})$.

\item \label{L:coker-sym2} If $G^{\sc}\cong G_1\times \ldots \times G_s$ is the decomposition of the universal cover $G^{\sc}$ of $\scr D(G)$ into simply connected quasi-simple groups, then 
 $\Bil^{s,\ev}(\Lambda(T_{\scr D(G)})\vert \Lambda(T_{G^{\ss}}))^{\scr W_G}$ is a finite index sub-lattice of 
$$
\left(\Sym^2 \Lambda^*(T_{G^{\sc}})\right)^{\mathscr W_G}=\Bil^{s,\ev}(\Lambda(T_{G^{\sc}}))^{\scr W_G}=\bbZ\langle b_1\rangle \oplus \ldots \oplus \bbZ\langle b_s\rangle,
$$ 
 where $b_i$ is the basic inner product of $G_i$.
 \item \label{L:coker-sym3} If $\scr D(G)$ is simply connected (i.e. $\scr D(G)=G^{\sc}$), then 
 $$
 \Bil^{s,\ev}(\Lambda(T_{\scr D(G)})\vert \Lambda(T_{G^{\ss}}))^{\scr W_G}=\Bil^{s,\ev}(\Lambda(T_{G^{\sc}}))^{\scr W_G}.
 $$
\end{enumerate}
\end{lem}
\begin{proof}
Part \eqref{L:coker-sym1} is proved in \cite[Prop. 2.4]{FV2}. Part \eqref{L:coker-sym2} follows from the fact that we have finite index inclusions of lattices
$$
\Lambda(T_{G^{\sc}})\hookrightarrow \Lambda(T_{\scr D(G)})\hookrightarrow \Lambda(T_{G^{\ss}})
$$
together with \cite[Lemma 2.2.1]{FV1}. Part \eqref{L:coker-sym3} follows from \cite[Lemma 4.3.4]{BH10}.

\end{proof}

A first case where \eqref{E:diag-Ab} is a push-out diagram is provided by the following 

\begin{teo}\label{T:PicBunGAss}
Assume that $G$ is a reductive group such that $\scr D(G)$ is simply connected and fix  $\delta\in\pi_1(G)$. 
There exists an exact sequence functorial in $S$ and in $G$
\begin{equation}\label{E:ex-Pic}
0\to \Pic \Bun_{G_{\ab}}^{\delta^{\ab}}(C/S)\xrightarrow{\ab_\sharp^*} \Pic \Bun_G^{\delta}(C/S)\to \left(\Sym^2 \Lambda^*(T_{G^{\sc}})\right)^{\mathscr W_G}\to 0
\end{equation}
together with a (non canonical) splitting functorial in $S$. 

In particular, the commutative diagram \eqref{E:diag-Ab} is a push-out diagram. 
\end{teo}
This Theorem was proved: for $G$ semisimple by Faltings \cite{Fa03}, for a fixed curve $C$ over $k=\ov k$ by Biswas-Hoffmann \cite{BH10} and for the universal family $\cC_{g,n}\to \cM_{g,n}$ for $g\geq 1$ by the authors in \cite[Thm. C(2)]{FV1} and \cite[Cor. 3.4]{FV2}.

\begin{proof}
The second assertion follows from the first one using Lemmas \ref{L:push} and \ref{L:coker-sym}, and the assumption that $\scr D(G)$ is simply connected. 
Hence, it is enough to prove the first assertion.

Let $\ov \eta$ be a geometric generic point of $S$. The family $\pi:C\to S$ sits in the following Cartesian diagram of families of curves
$$
\xymatrix{
C_{\ov \eta} \ar[r]\ar[d] \ar@{}[dr]|\square & C \ar[d]^{\pi} \ar[r] \ar@{}[dr]|\square & \cC_g\ar[d] \\
\ov \eta \ar[r] & S \ar[r] & \cM_g 
}
$$
The above diagram gives rise to the following commutative diagram with exact rows
\begin{equation}\label{E:diag-ab}
\xymatrix{
0\ar[r] & \RPic \Bun_{G_{\ab},g}^{\delta^{\ab}}\ar[r]^{\ab_\sharp^*} \ar@{^{(}->}[d]& \RPic \Bun_{G,g}^{\delta}\ar[r] \ar@{^{(}->}[d] & \left(\Sym^2 \Lambda^*(T_{G^{\sc}})\right)^{\mathscr W_G}\ar[r] \ar@{^{(}->}[d]^{\rho} &  0 \\
0\ar[r] & \RPic \Bun_{G_{\ab}}^{\delta^{\ab}}(C/S)\ar[r]^{\ab_\sharp^*(C/S)} \ar@{^{(}->}[d]& \RPic \Bun_G^{\delta}(C/S)\ar[r] \ar@{^{(}->}[d] & \cQ \ar[r] \ar@{->>}[d]^{\sigma} &  0 \\
0\ar[r] & \Pic \Bun_{G_{\ab}}^{\delta^{\ab}}(C_{\ov \eta})\ar[r]^{\ab_\sharp^*(C_{\ov \eta})} & \Pic \Bun_G^{\delta}(C_{\ov \eta})\ar[r] & \left(\Sym^2 \Lambda^*(T_{G^{\sc}})\right)^{\mathscr W_G}\ar[r] &  0 \\
}
\end{equation}
where 
\begin{itemize}
\item the first line is exact by \cite[Cor. 3.4]{FV2} for $g\geq 1$ and by Lemma below for $g=0$, where $\Bun^{\delta}_{G,g}:=\Bun_G^{\delta}(\cC_g/\cM_g)$ and similarly for $G^{\ab}$; 
\item $\cQ$ is the cokernel of $\ab_\sharp^*(C/S)$, so that the second line is exact by definition;
\item the third line is exact as it follows by applying \cite[Thm. 5.3.1]{BH10} to the morphism $\ab$ and using \cite[Prop. 5.2.11]{BH10} together with our hypothesis that $\scr D(G)$ is simple connected; 
\item the vertical morphisms in the left and central columns are injective by Fact  \ref{F:Pic-BunG}\eqref{F:Pic-BunG0};
\item the composition $\sigma\circ \rho$ is the identity, which implies that $\rho$ is injective and $\sigma$ is surjective.
\end{itemize}
It remains to show that $\sigma$ is injective (or, equivalently, that $\rho$ is surjective), and that the exact sequence \eqref{E:ex-Pic} admits a functorial splitting. 

To this aim, consider the surjective morphisms of tori $T_G\hookrightarrow G\xrightarrow{\ab} G^{\ab}$ and pick a section $\wt \iota: G^{\ab}\hookrightarrow T_G$. The composition $\iota: G^{\ab}\stackrel{\wt \iota}{\hookrightarrow} T_G\hookrightarrow G$ is a section of $\ab$. The induced morphism  
$$\iota_{\sharp}(C/S):\Bun_{G^{\ab}}^{\delta^{\ab}}(C/S)\hookrightarrow \Bun_{G}^{\delta}(C/S)$$
is a section of the morphism $\ab_{\sharp}(C/S)$. By taking the pull-back at the level of the Picard groups, we deduce that the inclusion 
$$
\ab_{\sharp}^*(C/S):  \Pic \Bun_{G_{\ab}}^{\delta^{\ab}}(C/S)\hookrightarrow \Pic \Bun_G^{\delta}(C/S)
$$
admits a splitting which is functorial in $S$, namely 
$$\iota_{\sharp}(C/S)^*: \Pic \Bun_G^{\delta}(C/S)\twoheadrightarrow \Pic \Bun_{G_{\ab}}^{\delta^{\ab}}(C/S).$$ 
This show that the exact rows in \eqref{E:diag-ab} admit compatible splittings, which implies that $\sigma$ is injective and we are done. 
\end{proof}

The following Lemma extends \cite[Cor. 3.4]{FV2} to genus $0$, for the reductive groups with simply connected derived subgroup. 

\begin{lem}\label{L:exact-g0}
Let $G$ be a reductive group such that $\scr D(G)$ is simply connected.
For every $n\geq 0$, there exists an exact sequence, functorial in $G$
\begin{equation}\label{E:exPic-g0}
0\to \Pic \Bun_{G_{\ab},0, n}^{\delta^{\ab}}\xrightarrow{\ab_\sharp^*} \Pic \Bun_{G,0, n}^{\delta}\to \left(\Sym^2 \Lambda^*(T_{G^{\sc}})\right)^{\mathscr W_G}\to 0.
\end{equation}
\end{lem}
\begin{proof}
We will use the notation and results from \cite[\S 5]{FV1}. Since the map $\pi_1(\ab):\pi_1(G)\xrightarrow{\cong} \pi_1(G^{\ab})=\Lambda(G^{\ab})$ is an isomorphism by the hypothesis that $\scr D(G)$ is simply connected (see \eqref{E:seq-pi1}), we can pick a lift $d\in \Lambda(T_G)$ of $\delta\in \pi_1(G)$ which is generic (see Definition/Lemma \ref{DL:gen-d}) and such that the following condition holds
\begin{equation}
\delta^{\ab} \text{ is $2$-divisible in }\Lambda(G^{\ab}) \Longrightarrow d  \text{ is $2$-divisible in }\Lambda(T_G). \tag{*}
\end{equation}
By \cite[Thm. 5.2]{FV1}, there exists a commutative diagram 
\begin{equation}\label{E:omega-g0}
\xymatrix{
\RPic(\Bun^{\delta^{\ab}}_{G^{\ab}, 0,n})\ar@{^{(}->}[r]^{\ab_\sharp^*} \ar@{^{(}->}[d]^{\omega_{G^{\ab}}^{\delta^{\ab}}} & \RPic(\Bun^{\delta}_{G, 0,n})\ar@{^{(}->}[r]^{\iota_{\sharp}^*}\ar@{^{(}->}[d]^{\omega_{G}^{\delta}} & 
\RPic(\Bun^d_{T_G, 0,n})\ar@{^{(}->}[d]^{\omega_{T_G}^{d}}\\
\Lambda^*(G^{\ab}) \ar@{^{(}->}[r] & \Omega_d^*(T_G) \ar@{^{(}->}[r] & \Lambda^*(T_G)
}
\end{equation}
where $\iota:T_G\hookrightarrow G$.

By the definition of $\Omega_d(T_G)$ and the fact that $d$ is generic (see Definition/Lemma \ref{DL:gen-d}), the contraction  homomorphism defines an isomorphism 
\begin{equation}\label{E:iso-contr}
\begin{aligned}
 (d^{\ss},-): \left(\Sym^2 \Lambda^*(T_{G^{\sc}})\right)^{\mathscr W_G} & \stackrel{\cong}{\longrightarrow} \frac{\Omega_d(T_G)}{\Lambda^*(G^{\ab})} \subseteq  \frac{\Lambda^*(T_G)}{\Lambda^*(G^{\ab})}=\Lambda^*(T_{G^{\sc}}), \\
 b &\mapsto b(d^{\ss}, -)
\end{aligned}
\end{equation}

Now, if $n\geq 1$ then the maps $\omega_{G^{\ab}}^{\delta^{\ab}}$ and $\omega_G^{\delta}$ are isomorphisms by \cite[Thm. 5.2(2)]{FV1} and we conclude by putting together \eqref{E:omega-g0} and \eqref{E:iso-contr}.
On the other hand, if $n=0$ then \cite[Thm. 5.2(2)]{FV1} gives that 
\begin{equation}\label{E:Im-diag}
\begin{sis}
& \Im(\omega_{G^{\ab}}^{\delta^{\ab}})=\{\chi \in \Lambda^*(G^{\ab})\: : \: \chi(\delta^{\ab})\in 2\bbZ\}, \\
& \Im(\omega_{G}^{\delta})=\{\chi \in \Omega^*_d(T_G)\: : \: \chi(d)\in 2\bbZ\}. \\
\end{sis}
\end{equation}
This implies that the commutative diagram of abelian groups is a pull-back diagram 
$$
\xymatrix{
\Lambda^*(G^{\ab}) \ar@{^{(}->}[r] &  \Omega^*_d(T_G) \\
 \Im(\omega_{G^{\ab}}^{\delta^{\ab}}) \ar@{^{(}->}[u] \ar@{^{(}->}[r]& \Im(\omega_{G}^{\delta}) \ar@{^{(}->}[u] 
}
$$
since if $\chi \in \Lambda^*(G^{\ab})$ then $\chi(d)=\chi(\delta^{\ab})$. Hence we get an injective homomorphism
\begin{equation}\label{E:alpha}
\alpha: \frac{\Lambda^*(G^{\ab})}{ \Im(\omega_{G^{\ab}}^{\delta^{\ab}})}\hookrightarrow  \frac{\Omega^*_d(T_G)}{\Im(\omega_{G}^{\delta})}
\end{equation}
Now if $\delta^{\ab}$ is not $2$-divisible then 
$$\frac{\Lambda^*(G^{\ab})}{ \Im(\omega_{G^{\ab}}^{\delta^{\ab}})}\cong \bbZ/2\bbZ\cong   \frac{\Omega^*_d(T_G)}{\Im(\omega_{G}^{\delta})}.$$
If, instead, $\delta^{\ab}$ is $2$-divisible then condition (*) implies that 
$$ 
\frac{\Lambda^*(G^{\ab})}{ \Im(\omega_{G^{\ab}}^{\delta^{\ab}})}=0= \frac{\Omega^*_d(T_G)}{\Im(\omega_{G}^{\delta})}.
$$
In any case, the morphism $\alpha$ of \eqref{E:alpha} is an isomorphism which implies, by the snake lemma applied to \eqref{E:Im-diag}, that 
$$
\frac{\Im(\omega_{G}^{\delta})}{\Im(\omega_{G^{\ab}}^{\delta^{\ab}})}\xrightarrow{\cong} \frac{\Omega^*_d(T_G)}{\Lambda^*(G^{\ab})}. 
$$
Hence we get the conclusion using again \eqref{E:omega-g0} and \eqref{E:iso-contr}.
\end{proof}

A second case where \eqref{E:diag-Ab} is a push-out diagram is described in the following 

\begin{teo}\label{T:PicBunGtf}
Let $G$ be a reductive group and fix  $\delta\in\pi_1(G)$.  If the family $C/S$ has genus $g(C/S)>0$ and it satisfies the following properties
\begin{enumerate}[(i)]
\item $\End(J_{C_{\ov \eta}}) =\bbZ$ and $\RPic(C/S)=\Pic_{C/S}(S)$;
\item $\RPic^0(C/S)$ is torsion-free;
\end{enumerate}
then  the commutative diagram \eqref{E:diag-Ab} is a push-out diagram. 
\end{teo}
This Theorem was proved for the universal family $\cC_{g,n}\to \cM_{g,n}$ with  $g\geq 1$ (which satisfy the assumptions, by Remarks \ref{R:eqRPic} and \ref{R:NSjump}\eqref{R:NSjump2}  and the Franchetta's conjecture)  in \cite[Thm. C(2)]{FV1}.
\begin{proof}
The proof follows the same lines of the proof of \cite[Thm. C(2)]{FV1}, as we now indicate. 

First of all, because of the assumptions,  the Picard group of $\Bun_{T_G}^d(C/S)$ (where $d\in \pi_1(T_G)$ is a fixed lift of $\delta\in \pi_1(G)$) is generated by tautological 
classes by Theorem \ref{T:all-taut}. Hence, using Theorem \ref{T:PicBunT>0}, we can define an action of the Weyl group $\scr W_G$ on $\Pic^d(\Bun_{T_G}^d(C/S))$ by setting 
$$
w\cdot d_{\wh \pi}(\cL_{\chi}(M)):=d_{\wh \pi}(\cL_{w\cdot \chi}(M)) \quad \text{ for any }\chi\in \Lambda^*(T) \: \text{ and  any } M\in \Pic(C/S),
$$
in such a way that the diagram \eqref{E:3seq} for $\RPic^d(\Bun_{T_G}^d(C/S))$  is $\scr W_G$-equivariant. 

Now, using the bottom exact sequence of \eqref{E:3seq} and arguing as in the proof of \cite[Lemma 5.3.3]{FV1} for the case $m=2$, we get that the subgroup of invariants of the above action of $\scr W_G$ sit in the following push-out diagram 
\begin{equation}\label{E:push-inv}
\xymatrix{
 \Sym^2 \Lambda^*(G^{\ab}) \ar@{^{(}->}[r]^{\Sym^2 \Lambda_{\ab}^*} \ar[d]^{\tau_{G^{\ab}}^{\delta^{\ab}}}& \left(\Sym^2 \Lambda^*(T_{G})\right)^{\mathscr W_G} \ar[d]^{\tau_{T_G}^d}\\
   \Pic \Bun_{G_{\ab}}^{\delta^{\ab}}(C/S)\ar@{^{(}->}[r]^{(\ab\circ \iota)_{\sharp}^*} &  \left(\Pic \Bun_{T_G}^{d}(C/S)\right)^{\scr W_G}
}
\end{equation}

Therefore, it remains to show that  the natural inclusion
\begin{equation}\label{E:incl-inv}
\psi:\left(\Pic \Bun_{T_G}^{d}(C/S)\right)^{\scr W_G} \hookrightarrow \Pic \Bun_G^{\delta}(C/S) 
\end{equation}
induced by the push-out diagram \eqref{E:push-inv} and the commutative diagram \eqref{E:push-diag}, is an isomorphism. This will be achieved in two steps:

\un{Step 1}: $\coker \psi$ is torsion.

Indeed, consider the two inclusions
$$
\ab_{\sharp}^*: \Pic \Bun_{G_{\ab}}^{\delta^{\ab}}(C/S)\stackrel{(\ab\circ \iota)_{\sharp}^*}{\hookrightarrow}  \left(\Pic \Bun_{T_G}^{d}(C/S)\right)^{\scr W_G}\stackrel{\psi}{\hookrightarrow} \Pic \Bun_G^{\delta}(C/S).
$$ 
We get an induced exact sequence 
$$
0\to \frac{\left(\Pic \Bun_{T_G}^{d}(C/S)\right)^{\scr W_G}}{\Pic \Bun_{G_{\ab}}^{\delta^{\ab}}(C/S)}\to \frac{\Pic \Bun_G^{\delta}(C/S)}{\Pic \Bun_{G_{\ab}}^{\delta^{\ab}}(C/S)}\to \coker \psi\to 0.
$$
By the push-out diagram \eqref{E:push-inv}, and using Lemmas \ref{L:push} and \ref{L:coker-sym}, we get that 
$$
\frac{\left(\Pic \Bun_{T_G}^{d}(C/S)\right)^{\scr W_G}}{\Pic \Bun_{G_{\ab}}^{\delta^{\ab}}(C/S)}\cong \bbZ^s,
$$
where $s$ is the number of quasi-simple factors of $G^{\sc}$. Moreover, arguing as in the proof of \cite[Lemma 5.3.4]{FV1}, we deduce that 
$$
\rk \frac{\Pic \Bun_G^{\delta}(C/S)}{\Pic \Bun_{G_{\ab}}^{\delta^{\ab}}(C/S)}\leq s.
$$
By putting everything together, we deduce that $\coker \psi$ is a finite group.

\un{Step 2}: $\coker \psi$ is torsion-free.

Indeed, the natural inclusion of the subgroup of $\scr W_G$-invariants of $\Pic \Bun_{T_G}^{d}(C/S)$ factors as 
$$\left(\Pic \Bun_{T_G}^{d}(C/S)\right)^{\scr W_G} \stackrel{\psi}{\hookrightarrow} \Pic \Bun_G^{\delta}(C/S) \stackrel{\iota_\sharp^*}{\hookrightarrow} \Pic \Bun_{T_G}^{d}(C/S),$$
where $\iota_{\sharp}^*$ is injective by Fact \ref{F:Pic-BunG}\eqref{F:Pic-BunG1}.
Therefore $\coker \psi$ injects into the quotient 
$$
Q:=\frac{\Pic \Bun_{T_G}^{d}(C/S)}{\left(\Pic \Bun_{T_G}^{d}(C/S)\right)^{\scr W_G}}=\frac{\RPic \Bun_{T_G}^{d}(C/S)}{\left(\RPic \Bun_{T_G}^{d}(C/S)\right)^{\scr W_G}}
$$
Because of our three assumptions on the family $C/S$, Theorems \ref{T:all-taut} and \ref{T:PicBunT>0} imply that $\RPic \Bun_{T_G}^{d}(C/S)$ is torsion-free. Hence, the quotient $Q$ is torsion-free by \cite[Lemma 5.3.6]{FV1}, which implies that $\coker \psi$ is torsion-free.

\end{proof}

We now explain how to obtain two alternative presentations of the relative Picard group of $\Bun_G^{\delta}(C/S)$, for a family of curves $C/S$ of positive genus and such that 
 $\End(J_{C_{\ov \eta}}) =\bbZ$ and $\RPic(C/S)=\Pic_{C/S}(S)$, and assuming that \eqref{E:push-diag} is a push-out diagram.

We first need the following 

\begin{defin}\label{D:NS}\cite[Cor. 2.6, Def. 3.9]{FV2}
Let $G$ be a reductive group with maximal torus $T_G$ and Weyl group $\scr W_G$, and fix $\delta\in \pi_1(G)$.
\begin{enumerate}
\item 
$
 \Bil^{s,\scr D-\ev}(\Lambda(T_G))^{\scr W_G}:=\left\{b\in  \Bil^{s}(\Lambda(T_G))^{\scr W_G}\: : b_{|\Lambda(T_{\scr D(G)})\times \Lambda(T_{\scr D(G)})} \text{ is even}\right\}.
$
\item The \emph{Neron-Severi} group of $\Bun_G^{\delta}$ is  the subgroup 
$$\NS(\Bun_G^{\delta})\subseteq   \frac{\Lambda^*(T_G)}{\Lambda^*(T_{G^{\ad}})} \oplus   \Bil^{s,\scr D-\ev}(\Lambda(T_G))^{\scr W_G}$$
consisting of all the elements $([\chi], b)$ such that 
\begin{equation}\label{E:NS-BunG}
\left[\chi_{|\Lambda(T_{\scr D(G)})}\right]=b(\delta\otimes -)_{|\Lambda(T_{\scr D(G)})}:=\left[b(d\otimes -)_{|\Lambda(T_{\scr D(G)})} \right] \in \frac{\Lambda^*(T_{\scr D(G)})}{\Lambda^*(T_{G^{\ad}})},
\end{equation}
 where $d\in \Lambda(T_G)$ is any lift of $\delta$.
\end{enumerate}
\end{defin}
For a more detailed study of the Neron-Severi group $\NS(\Bun_G^{\delta})$, see \cite[Prop. 3.11]{FV2}.

\begin{teo}\label{T:altpres}
Assume that the family $C/S$ has genus $g(C/S)>0$ and it satisfies the following conditions:
\begin{enumerate}[(i)]
\item $\End(J_{C_{\ov \eta}}) =\bbZ$ and $\RPic(C/S)=\Pic_{C/S}(S)$;
\item either $\RPic^0(C/S)$ is torsion-free or $\scr D(G)$ is simply connected.
\end{enumerate}
Then there exists a commutative diagram, functorial in $S$ and in $G$,
\begin{equation}\label{E:diagpres}
\xymatrix{
0 \ar[r] &  \Lambda^*(G^{\ab})\otimes \RPic^0(C/S)  \ar[r]^(0.55){j_G^{\delta}} \ar@{^{(}->}[d] & \RPic\Bun_{G}^{\delta}(C/S)\ar[r]^{\omega_G^{\delta}\oplus \gamma_G^{\delta}}  \ar@{=}[d]  &\NS(\Bun_G^{\delta})\ \ar@{->>}[d]^{\res_{G}^{\NS}} \\
0  \ar[r] &  \Lambda^*(G^{\ab})\otimes \RPic(C/S)  \ar[r]^(0.55){i_G^{\delta}} &  \RPic\Bun_{G}^{\delta}(C/S)  \ar[r]^(0.45){\gamma_G^{\delta}}  & \Bil^{s,\scr D-\ev}(\Lambda(T_G))^{\scr W_G}  \ar[r]  & 0\\
}
\end{equation}
where the rows are exact, the left vertical morphism is injective and the right vertical morphism is surjective.

Moreover, the image of $\omega_G^{\delta}\oplus \gamma_G^{\delta}$ is equal to 
\begin{equation*}
 \Im(\omega_G^{\delta}\oplus \gamma_G^{\delta})=
  \left\{([\chi], b)\in \NS(\Bun_G^{\delta}) \: : \: 
 \begin{aligned}
 &  \left[\chi(x)-b(\delta\otimes x)\right]+(g-1)b(x\otimes x) \\
 & \text{ is divisible by } \delta(C/S),  \text{ for any } x\in \Lambda(T_G) 
 \end{aligned}
 \right\}
 \end{equation*}
\end{teo}
This Theorem was proved for the universal family $\cC_{g,n}\to \cM_{g,n}$ for $g\geq 1$ (and an arbitrary reductive group) in \cite[Thm. 1.1]{FV2} and it can be deduced  for a fixed curve $C$ over $k=\ov k$ (in the case $g(C)>0$, $\End(J_{C_{\ov\eta}})=\bbZ$  and $\scr D(G)$ simply connected) from  \cite[Thm. 5.3.1]{BH10}. 
\begin{proof}
Observe that our assumptions imply that \eqref{E:push-diag} is a push-out diagram (by Theorems \ref{T:PicBunGAss} and \ref{T:PicBunGtf}), and that $\Pic \Bun_{G^{\ab}}^{\delta^{\ab}}(C/S)= \Pic^{\taut} \Bun_{G^{\ab}}^{\delta^{\ab}}(C/S)$ by Theorem \ref{T:all-taut}.

Let us first describe the maps appearing in the above  diagram \eqref{E:diagpres}:

$\bullet$ the map $\iota_G^{\delta}$ is defined by 
$$
i_G^{\delta}(\chi\otimes M):=\langle \cL_{\chi}, M \rangle_{\wh \pi} 
$$
and $j_G^{\delta}$ is the restriction of $i_G^{\delta}$.

$\bullet$ the morphism $\omega_G^{\delta}$ is induced by the weight homomorphism 
$$
\w_G^{\delta}:\Pic \Bun_G^{\delta}(C/S)\to \Lambda^*(\scr Z(G))=\frac{\Lambda^*(T_G)}{\Lambda^*(T_{G^{\ad}})}
$$
coming from the exact sequence \eqref{E:seqLeray}.

$\bullet$ the morphism $\gamma_G^{\delta}$ is uniquely determined, using that the diagram \eqref{E:push-diag} is a push-out, by 
\begin{enumerate}[(i)]
\item \label{D:gammaG1} The composition $\gamma_G^{\delta}\circ  \tau_G^{\delta}$ is the inclusion
$$
\alpha: \left(\Sym^2\Lambda^*(T_G)\right)^{\scr W_G}\cong  \Bil^{s,\ev}(\Lambda(T_G))^{\scr W_G}  \hookrightarrow \Bil^{s, \scr D-\ev}(\Lambda(T_G))^{\scr W_G}.
$$ 
\item\label{D:gammaG2} The composition $\gamma_G^{\delta}\circ \ab_\#^*$ is equal to the following composition 
$$
 \RPic^{\taut} \Bun_{G^{\ab}}^{\delta^{\ab}}(C/S)
 \xrightarrow{\gamma_{G^{\ab}}^{\delta^{\ab}}} (\Lambda^*(G^{\ab})\otimes \Lambda^*(G^{\ab}))^s \cong \Bil^{s}(\Lambda(G^{\ab})) \xrightarrow{B_{\ab}^*}   \Bil^{s, \scr D-\ev}(\Lambda(T_G))^{\scr W_G},
$$
where $\gamma_{G^{\ab}}^{\delta^{\ab}}$ is the homomorphism of Lemma \ref{L:gammaT} and $B_{\ab}^*$ is obtained by pull-back.
\end{enumerate}

Now the proof is deduced from Theorems \ref{T:PicBunT>0}, \ref{T:all-taut} and  \ref{T:PicBunGAss}, arguing as in the proofs of \cite[Thm. 3.6, Thm. 3.12]{FV2}.
\end{proof}

We end this subsection by describing the relative Picard group of $\Bun_G^{\delta}(C/S)$, for a family of curves $C/S$ of genus zero and an arbitrary reductive group $G$.  
We first need the following 

\begin{defin}(see \cite[Def. 5.2.1]{BH10})
Let $G$ be a reductive group and let $\delta\in \pi_1(G)$.
Define 
$$
\NS \Bun_G^{\delta}(\bbP^1):=\left\{(\chi, b): l+b(d^{\ss},-)\in \Lambda^*(T_G)\right\}\subseteq \Lambda^*(\scr Z(G))\times \left(\Sym^2 \Lambda^*(T_{G^{\sc}})\right)^{\scr W_G},
$$
where $d^{\ss}$ is the image in $\Lambda(T_{G^{\ss}})$ of any lift $d\in \Lambda(T_G)$ of $\delta$, and  $l+b(d^{\ss},-)$ is interpreted as an element of 
$$\Lambda^*(\scr Z(G))_{\bbQ}\oplus \Lambda^*(T_{G^{\sc}})_{\bbQ}=\Lambda^*(T_G)_{\bbQ}.$$
\end{defin}
It is easily checked that the above definition does not depend on the choice of the lifting $d$ of $\delta$ (see \cite[Lemma 5.2.2]{BH10}) and that it is functorial in $G$ (see \cite[Def. 5.2.7]{BH10}).

\begin{teo}\label{T:PicG=0}
Let $C\to S$ be a family of curves of genus $g=0$ on a regular and integral quotient stack. Let $G$ be a reductive group, fix $\delta\in \pi_1(G)$ and set $\delta^{\ab}:=\pi_1(\ab)(\delta)\in \pi_1(G^{\ab})$.
Fix a lift $d\in \Lambda(T_G)$ of $\delta$ and denote by $d^{\ss}$ its image in $\Lambda(T_{G^{\ss}})$. 

There exists an injective homomorphism, functorial in $S$ and $G$:
$$
c_G^{\delta}(C/S)=c_G^{\delta}: \RPic \Bun_G^{\delta}(C/S)\to \NS \Bun_G^{\delta}(\bbP^1)
$$
such that 
\begin{enumerate}
\item \label{T:PicG=01} If $G=T$ is a torus then $c_G^{\delta}$ coincides with the weight homomorphism;
\item \label{T:PicG=02} the image of $c_G^{\delta}$ is equal to 
\begin{equation}\label{E:ImcTG}
\Im(c_G^{\delta})=
\begin{cases}
\NS \Bun_G^{\delta}(\bbP^1) & \text{ if } C/S  \text{ is Zariski locally trivial, } \\
\{(l,b)\in \NS \Bun_G^{\delta}(\bbP^1) : l(\delta^{\ab})+b(d^{\ss}, d^{\ss}) \text{ is even} \}& \text{ if }   C/S  \text{ is not Zariski locally trivial.} 
\end{cases}
\end{equation}
In particular, if $C/S$ is not Zariski locally trivial, then $\Im(c_G^{\delta})$ is an index two subgroup of $\NS \Bun_G^{\delta}(\bbP^1)$.
\end{enumerate}
\end{teo}
Note that the second subgroup in \eqref{E:ImcTG} is well-defined since if $\wt d\in \Lambda(T_G)$ is another lift of $\delta$, then its image $\wt d^{\ss}$ in $\Lambda(T_{G^{\ss}})$ is such that $\wt d^{\ss}=d^{\ss}+\epsilon$ for some $\epsilon \in \Lambda(T_{G^{\sc}})$, and we have 
$$
b(\wt d^{\ss}, \wt d^{\ss})=b(d^{\ss},d^{\ss})+2b(d^{\ss},\epsilon)+b(\epsilon,\epsilon)\in b(d^{\ss},d^{\ss})+2\bbZ,
$$
using that $b(\epsilon,\epsilon)\in 2\bbZ$ since $b$ is an even symmetric bilinear form on $\Lambda(T_{G^{\sc}})$, and that $b(d^{\ss}, \epsilon)\in \bbZ$ by \cite[Lemma 4.3.4]{BH10}.  
  
\begin{proof}
Note that $C/S$ is Zariski locally trivial if and only if $C/S$ admits a section. Therefore, for a choice of a geometric generic point $\ov \eta$ of $S$, we get the following diagram
\begin{equation}\label{E:Cart-CS}
\xymatrix{
C_{\ov \eta}\ar[r] \ar[d] \ar@{}[dr]|\square& C\ar[r] \ar[d] \ar@{}[dr]|\square & \cC_{0,n} \ar[d] \\
\ov \eta \ar[r] & S\ar[r]  & \cM_{0,n} \\
}
\end{equation}
with Cartesian squares, where 
$$n=
\begin{cases}
1 & \text{if }C/S \text{ is Zariski locally trivial,} \\
0 & \text{if }C/S \text{ is not Zariski locally trivial.}
\end{cases}
$$

The above diagram induces the following homomorphisms, functorial in $S$ and in $G$:
\begin{equation}\label{E:seq-hom}
\RPic\Bun_{G,0,n}^{\delta} \hookrightarrow \RPic \Bun_G^{\delta}(C/S) \hookrightarrow \Pic \Bun_G^{\delta}(\bbP^1_{\ov \eta})\xrightarrow[\cong]{c_G^{\delta}(\bbP^1_{\ov \eta})} \NS\Bun_G^{\delta}(\bbP^1),
\end{equation}
where the first two pullback homomorphisms are injective by Fact \ref{F:Pic-BunG}\eqref{F:Pic-BunG0} (note that $\ov \eta$ is also a geometric generic point of $\cM_{0,n}$), and the last isomorphism is the one of \cite[Thm. 5.3.1(iii)]{BH10}. Therefore, we get our desired injective homomorphism, functorial in $S$ and in $G$
$$
c_G^{\delta}(C/S): \RPic \Bun_G^{\delta}(C/S) \hookrightarrow \Pic \Bun_G^{\delta}(\bbP^1_{\ov \eta})\xrightarrow[\cong]{c_G^{\delta}(\bbP^1_{\ov \eta})} \NS\Bun_G^{\delta}(\bbP^1)
$$
which coincides with the weight homomorphism if $G$ is a torus, because this is the case for the isomorphism $c_G^{\delta}(\bbP^1_{\ov \eta})$ (see the proof of \cite[Thm. 5.3.1(iii)]{BH10}).

It remains to prove part \eqref{T:PicG=02}, by distinguishing the two cases $C/S$ Zariski locally trivial or not. 

If $C/S$ is Zariski locally trivial (and hence $n=1$ in the above diagram), then we get that $c_G^{\delta}(C/S)$ is surjective since the composite homomorphism in  \eqref{E:seq-hom} 
$$
c_G^{\delta}(\cC_{0,1}/\cM_{0,1}): \RPic\Bun_{G,0,1}^{\delta} \to  \NS\Bun_G^{\delta}(\bbP^1)
$$
is surjective by \cite[Thm. 5.0.2]{FV1}.

If $C/S$ is not Zariski locally trivial (and hence $n=0$ in the above diagram), then we get that 
\begin{equation*}\label{E:incl1}
\{(l,b)\in \NS \Bun_G^{\delta}(\bbP^1) : l(\delta^{\ab})+b(d^{\ss}, d^{\ss}) \text{ is even}\}\subseteq \Im(c_G^{\delta}(C/S)),
\end{equation*}
since the same is true for the image of $c_G^{\delta}(\cC_0/\cM_0)$ by \cite[Thm. 5.0.2]{FV1}. In order to prove the other inclusion, consider the following commutative diagram 
(for a chosen lift $d\in \Lambda(T_G)$ of $\delta$)
\begin{equation*}
\xymatrix{ 
 \RPic \Bun_G^{\delta}(C/S) \ar[r]^{\iota_{\sharp}^*} \ar[d]^{c_G^{\delta}(C/S)}&  \RPic \Bun_{T_G}^{d}(C/S)\ar[d]^{c_{T_G}^{d}(C/S)}\\
 \NS\Bun_G^{\delta}(\bbP^1) \ar[r]^(0.4){\iota_{\sharp}^{\NS,*}} & \NS\Bun_{T_G}^{d}(\bbP^1)=\Lambda^*(T_G)\\
}
\end{equation*}
where the bottom map is defined by (see \cite[Def. 5.2.5]{BH10})
$$
\begin{aligned}
\iota_{\sharp}^{\NS,*}: \NS\Bun_G^{\delta}(\bbP^1) & \to \NS\Bun_{T_G}^{d}(\bbP^1)=\Lambda^*(T_G)\\
 (l,b)  &\mapsto  l+b(d^{\ss},-)
 \end{aligned}
$$ 
The above commutative diagram, together with Theorem \ref{T:PicBunT-0}\eqref{T:PicBunT-02}, implies that 
$$
\Im(c_G^{\delta}(C/S))\subseteq (\iota_{\sharp}^{\NS,*})^{-1}(\Im(c_{T_G}^d(C/S)))=\{(l,b)\in \NS \Bun_G^{\delta}(\bbP^1) : l(\delta^{\ab})+b(d^{\ss}, d^{\ss}) \text{ is even}\},
$$
and we are done. 
\end{proof}

\begin{cor}\label{C:PicG=0}
Same assumptions of Theorem \ref{T:PicG=0}. 
There exists an exact sequence functorial in $S$ and in $G$
\begin{equation}\label{E:ex-Pic-g0}
0\to \Pic \Bun_{G_{\ab}}^{\delta^{\ab}}(C/S)\xrightarrow{\ab_\sharp^*} \Pic \Bun_G^{\delta}(C/S)\xrightarrow{p_G^{\delta}}
\left\{\begin{aligned} b\in \left(\Sym^2 \Lambda^*(T_{G^{\sc}})\right)^{\mathscr W_G}: \\ b^{\bbQ}(d^{\ss},-)_{|\Lambda(T_{\scr D(G)})}\text{ is integral }\end{aligned}\right\}
\end{equation}
where the  last homomorphism $p_G^{\delta}$ has image of index at most two and it is surjective  if either $\scr D(G)$ is simply connected or $C/S$ is Zariski locally trivial. 
\end{cor}
Note that if $\scr D(G)$ is simply connected then the last group of \eqref{E:ex-Pic-g0} is equal to $ \left(\Sym^2 \Lambda^*(T_{G^{\sc}})\right)^{\mathscr W_G}$ by Lemma \ref{L:coker-sym}\eqref{L:coker-sym3}.

\begin{proof}
The above Theorem \ref{T:PicG=0} implies that the following diagram 
\begin{equation*}
\xymatrix{ 
 \RPic \Bun_{G^{\ab}}^{\delta^{\ab}}(C/S) \ar@{^{(}->}[r]^{\ab_{\sharp}^*} \ar@{^{(}->}[d]^{\w_{G^{\ab}}^{\delta^{\ab}}(C/S)}&  \RPic \Bun_G^{\delta}(C/S)\ar@{^{(}->}[d]^{c_{G}^{\delta}(C/S)}\\
 \NS\Bun_{G^{\ab}}^{\delta^{\ab}}(\bbP^1)=\Lambda^*(G^{\ab}) \ar@{^{(}->}[r]^(0.6){\ab_{\sharp}^{\NS,*}} & \NS\Bun_{G}^{\delta}(\bbP^1)
}
\end{equation*}
is Cartesian, where the map $\ab_{\sharp}^{\NS,*}$ sends $\chi$ into $(\chi, 0)$. Therefore, we get that the cokernel of $\ab_{\sharp}^*$ injects into the cokernel of $\ab_{\sharp}^{\NS,*}$, which is equal (by \cite[Prop. 5.2.11]{BH10}) to 
$$ 
\coker(\ab_{\sharp}^{\NS,*})=\left\{ b\in \left(\Sym^2 \Lambda^*(T_{G^{\sc}})\right)^{\mathscr W_G}: b^{\bbQ}(d^{\ss},-)_{|\Lambda(T_{\scr D(G)})}\text{ is integral }\right\}.
$$
This proves the existence of the exact sequence \eqref{E:ex-Pic-g0} as well as of the following exact sequence
$$
0\to \coker \w_{G^{\ab}}^{\delta^{\ab}}(C/S)\to \coker c_{G}^{\delta}(C/S)\to \coker p_G^{\delta} \to 0.
$$
Since $\coker c_{G}^{\delta}(C/S)$ is a finite group of order at most two by Theorem \ref{T:PicG=0} and it is trivial if $C/S$ is Zariski locally trivial,  we deduce that the same is true for $\coker p_G^{\delta}$. Moreover, $\coker p_G^{\delta}$ is trivial if $\scr D(G)$ is simply connected by 
Theorem \ref{T:PicBunGAss}.
\end{proof}

\section{The rigidification of $ \Bun_{G}^{\delta}(C/S)$ by the center}\label{S:rigid}

The aim of this section is to study, for a reductive group $G$,  the $\scr Z(G)$-gerbe (see \eqref{E:rigid})
\begin{equation*}
\nu_G^{\delta}=\nu_G^{\delta}(C/S):\Bun_{G}^{\delta}(C/S)\to \Bun_{G}^{\delta}(C/S)\fatslash \scr Z(G):=\Bunr_{G}^{\delta}(C/S),
\end{equation*}
obtained  by rigidifying the stack $\Bun_{G}^{\delta}(C/S)$ by the center $\scr Z(G)$ of the reductive group $G$.


From the Leray spectral sequence 
$$
E_{p,q}^2=H^p(\Bunr_G^{\delta}(C/S), R^q(\nu_{G}^{\delta})_*(\Gm))\Rightarrow H^{p+q}(\Bun_G^{\delta}(C/S), \Gm), 
$$
and using that $(\nu_G^{\delta})_*(\Gm)=\Gm$ and that $R^1(\nu_{G}^{\delta})_*(\Gm)$ is the constant sheaf $\Lambda^*(\scr Z(G))=\Hom(\scr Z(G),\Gm)$, we get the exact sequence 
\begin{equation}\label{E:seqLeray}
\begin{aligned}
&  \Pic(\Bunr_G^{\delta}(C/S))\stackrel{(\nu_G^{\delta})^*}{\hookrightarrow} \Pic(\Bun_G^{\delta}(C/S)) \xrightarrow{\w_G^{\delta}} \Lambda^*(\scr Z(G)):=\Hom(\scr Z(G),\Gm) \\
&    \xrightarrow{\obs_G^{\delta}} H^2(\Bunr_G^{\delta}(C/S), \Gm)\xrightarrow{(\nu_G^{\delta})^*}  H^2(\Bun_G^{\delta}(C/S), \Gm),
\end{aligned}
\end{equation}
and an analogue one in which the Picard groups $\Pic$ are replaced by the relative Picard groups $\RPic$.

Recall (see \cite[\S 5]{FV2}) the following geometric  descriptions of the \emph{weight homomorphism} $\w_G^{\delta}$ and the \emph{obstruction homomorphism} $\obs_G^{\delta}$:
\begin{itemize}
\item given a line bundle $\cL$ on $\Bun_G^{\delta}(C/S)$, the character $\w_G^{\delta}(\cL)\in \Lambda^*(\scr Z(G))$ is such that, for any  $\cE:=(E\to C_V)\in \Bun_G^{\delta}(C/S)(V)$, we have the factorization
$$
\w_G^{\delta}(V):\scr Z(G)(V) \hookrightarrow \Aut_{\Bun_G^{\delta}(C/S)(V)}(\cE)\xrightarrow{\Aut(\cL_{V})} \Aut_{\cO_V}(\cL_V(\cE))=\Gm(V),
$$
where the first homomorphism is given by the canonical action of $\scr Z(G)$ on every $G$-gerbe, and the second homomorphism is induced by the functor of groupoids 
$$\cL_V:\Bun_G^{\delta}(C/S)(V)\to \{\text{Line bundles on } V\}$$ 
determined by $\cL$. 
\item given any character $\lambda\in \Hom(\scr Z(G),\Gm)$, the element $\obs_G^{\delta}(\lambda)$ is the class in $H^2(\Bunr_G^{\delta}(C/S), \Gm)$ of the $\Gm$-gerbe $\lambda_*(\nu_G^{\delta})$ obtained by push-forwarding the $\scr Z(G)$-gerbe $\nu_G^{\delta}$ along $\lambda$. 
\end{itemize}

Recall also (see \cite[\S 2.1]{FV2})  that the character group of the center $\scr Z(G)$ is equal to
$$\Lambda^*(\scr Z(G))=\frac{\Lambda^*(T_G)}{\Lambda^*(T_{G^{\ad}})}$$
and it sits into the following exact sequence induced by the vertical exact sequence of \eqref{E:cross}:
\begin{equation}\label{E:LambdaZG}
0\to \Lambda^*(G^{\ab})\xrightarrow{\ov{\Lambda_{\ab}^*}} \Lambda^*(\scr Z(G))=\frac{\Lambda^*(T_G)}{\Lambda^*(T_{G^{\ad}})} \xrightarrow{\ov{\Lambda_{\scr D}^*}} \Lambda^*(\scr Z(\scr D(G)))=\frac{\Lambda^*(T_{\scr D(G)})}{\Lambda^*(T_{G^{\ad}})} \to 0.
\end{equation}

Our goal is to determine the (relative) Picard group of the rigidification $\Bunr_{G}^{\delta}(C/S)$ and the group 
\begin{equation}\label{E:obs-wei}
 \coker(\w_G^{\delta}) \cong \Im(\obs_G^{\delta})=\ker (\nu_G^{\delta})^*
\end{equation}
which is an  obstruction to the triviality of the $\scr Z(G)$-gerbe $\nu_G^{\delta}$. 

With this aim, we need to recall two crucial homomorphisms from \cite{FV2}. 

\begin{defin}\label{D:evGsc}(see \cite[Definition/Lemma 2.8, Def. 5.3]{FV2})
Let $G$ be a reductive group with fixed maximal torus $T_G$ and Weyl group $\scr W_G$, and fix $\delta\in \pi_1(G)$. 
\begin{enumerate}
\item There is a well-defined homomorphism (called \emph{evaluation homomorphism})
\begin{equation}\label{E:evGsc}
\begin{aligned}
\ev_{\scr D(G)}^{\delta}:\Bil^{s,\ev}(\Lambda(T_{\scr D(G)})\vert \Lambda(T_{G^{\ss}}))^{\scr W_G} & \longrightarrow \frac{\Lambda^*(T_{\scr D(G)})}{\Lambda^*(T_{G^{\ad}})}, \\
b &\mapsto b(\delta^{\ss}\otimes -):=[b^{\bbQ}(d^{\ss}\otimes -)], 
\end{aligned}
\end{equation}
where $d^{\ss}\in \Lambda(T_{G^{\ss}})$ is any lift of $\delta^{\ss}:=\pi_1(\ss)(\delta)\in \pi_1(G^{\ss})=\frac{\Lambda(T_{G^{\ss}})}{\Lambda(T_{G^{\sc}})}$ and $d^{\bbQ}$ is the rational extension of $b$ to $\Lambda(T_{G^{\ss}})$. 
\item The \emph{Neron-Severi}  group of $\Bunr_G^{\delta}$ is the subgroup 
$$\NS(\Bunr_{G}^{\delta})\subset \Bil^{s,\scr D-\ev}(\Lambda(T_G))^{\scr W_G}$$
 consisting of those $b\in \Bil^{s,\scr D-\ev}(\Lambda(T_G))^{\scr W_G}$ such that 
\begin{equation}\label{E:NS-rig}
0=b(\delta\otimes -)_{|\Lambda(T_{\scr D(G)})}:=\left[b(d\otimes -)_{|\Lambda(T_{\scr D(G)})} \right] \in  \frac{\Lambda^*(T_{\scr D(G)})}{\Lambda^*(T_{G^{\ad}})},
\end{equation}
where $d\in \Lambda(T_G)$ is any lift of $\delta$.
\end{enumerate}
\end{defin}

Note that there is an injective homomorphism 
$$
\begin{aligned}
\nu^{\delta, \NS}: \NS(\Bunr_{G}^{\delta}) & \hookrightarrow \NS(\Bun_G^{\delta})\\
b & \mapsto ([0],b).
\end{aligned}
$$ 
For a more detailed description of  $\NS(\Bunr_G^{\delta})$, see \cite[Prop. 5.4]{FV2}.

\begin{ese}\label{R:coker-ev}
If $\scr D(G)$ is simply connected, i.e. $\scr D(G)=G^{\sc}$, then the cokernel of $\ev_{G^{\sc}}^{\delta}$ can be computed explicitly as it follows. 

\begin{enumerate}
\item If $G^{\sc}=G_1^{\sc}\times \cdots \times G_r^{\sc}$ is the decomposition of $G^{\sc}$ into simply connected and almost simple algebraic groups and $\delta=(\delta_1,\ldots, \delta_r)\in \pi_1(G^{\ad})=\pi_1(G_1^{\ad})\times \cdots \times \pi_1(G_r^{\ad})$, then 
$$
\coker \ev_{G^{\sc}}^{\delta}\cong \coker \ev_{G_1^{\sc}}^{\delta_1}\times \cdots \times \coker \ev_{G_r^{\sc}}^{\delta_r}.
$$
 \item  If $G^{\sc}$ is simply connected and almost simple, then the cokernel of $\ev_{G^{\sc}}^{\delta}$ is equal to (see \cite[\S 7]{FV2}):  
\begin{enumerate}
\item If $G^{\sc}=\SL_n$ (and hence $G^{\ad}=\PSL_n$) then, for any $\delta\in \pi_1(\PSL_n)=\bbZ/n\bbZ$, we have that 
$$
\coker \ev_{\SL_n}^{\delta}=\frac{\bbZ}{\gcd(n,\delta)\bbZ}.  
$$
\item If $G^{\sc}=\Spin_{2n+1}$ (and hence $G^{\ad}=\SO_{2n+1}$) then, for any $\delta\in \pi_1(\SO_{2n+1})=\bbZ/2\bbZ$, we have that 
$$
\coker \ev_{\Spin_{2n+1}}^{\delta}=\bbZ/2\bbZ.  
$$
\item If $G^{\sc}=\Sp_{2n}$ (and hence $G^{\ad}=\PSp_{2n}$) then, for any $\delta\in \pi_1(\PSp_{2n})=\bbZ/2\bbZ$, we have that 
$$
\coker \ev_{\Sp_{2n}}^{\delta}=
\begin{cases}
\bbZ/2\bbZ & \text{ if either }\delta= 0 \text{ or } n \text{ is even}, \\
0 & \text{ otherwise.}
\end{cases}  
$$
\item If $G^{\sc}=\Spin_{2n}$ (and hence $G^{\ad}=\PSO_{2n}$) then, for any 
$$\delta\in \pi_1(\PSO_{2n})=
\begin{cases}
\bbZ/4\bbZ & \text{ if } n \: \text{ is odd,}\\
\bbZ/2\bbZ\times \bbZ/2\bbZ & \text{ if }n \: \text{ is even,} 
\end{cases}
$$ 
we have that 
$$
\coker \ev_{\Spin_{2n}}^{\delta}=
\begin{cases}
\bbZ/4\bbZ & \text{ if } \delta=0 \text{ and }n \: \text{ is odd,}\\
\bbZ/2\bbZ\times \bbZ/2\bbZ & \text{ if } \delta=0  \:\text{ and } n  \text{ is even,} \\
\bbZ/2\bbZ & \text{ if } \ord(\delta)=2,\\
0 & \text{ if } \ord(\delta)=4.
\end{cases}  
$$
\item If $G^{\sc}=\bbE_{6}^{\sc}$ then, for any $\delta\in \pi_1(\bbE_6^{\ad})=\bbZ/3\bbZ$, we have that 
$$
\coker \ev_{\bbE_6^{\sc}}^{\delta}=
\begin{cases}
\bbZ/3\bbZ & \text{ if } \delta=0, \\
0 & \text{ if } \delta\neq 0.
\end{cases}  
$$
\item If $G^{\sc}=\bbE_{7}^{\sc}$ then, for any $\delta\in \pi_1(\bbE_7^{\ad})=\bbZ/2\bbZ$, we have that 
$$
\coker \ev_{\bbE_7^{\sc}}^{\delta}=
\begin{cases}
\bbZ/2\bbZ & \text{ if } \delta= 0, \\
0 & \text{ if } \delta\neq 0.
\end{cases}  
$$
\item If $G^{\sc}$ is of type $\bbE_8$ or $\bbF_4$ or $\bbG_2$ then 
$$
\coker \ev_{G^{\sc}}^{\delta}=0.
$$
\end{enumerate}
\end{enumerate}
\end{ese}

We first describe the (relative) Picard group of the rigidification $\Bunr_{G}^{\delta}(C/S)$ for families of positive genus, under certain assumptions.

\begin{teo}\label{T:PicRig>0}
Let $C\to S$ be a family of curves of genus $g>0$ on a regular and integral quotient stack with geometric generic fiber $C_{\ov \eta}$.
Assume that  the two following two conditions hold true:
\begin{enumerate}[(i)]
\item  $\End(J_{C_{\ov \eta}}) =\bbZ$ and $\RPic(C/S)=\Pic_{C/S}(S)$;
\item either $\RPic^0(C/S)$ is torsion-free or $\scr D(G)$ is simply connected.
\end{enumerate}
Then, for any fixed $\delta\in \pi_1(G)$, we have the following exact sequence, functorial in $S$ and in $G$:
\begin{equation}\label{E:Pic-rig1}
0 \to \Lambda^*(G^{\ab})\otimes \RPic^0(C/S)  \xrightarrow{\ov{j_G^{\delta}}} \RPic(\Bunr_G^{\delta}(C/S))\xrightarrow{\ov{\gamma_G^{\delta}}} \NS(\Bunr_G^{\delta}),
\end{equation}
where $\ov{j_G^{\delta}}$ and $\ov{\gamma_G^{\delta}}$ are uniquely determined by 
$$(\nu_G^{\delta})^*\circ \ov{j_G^{\delta}}=j_G^{\delta}\quad \text{ and } \quad \nu_G^{\delta,\NS} \circ \ov{\gamma_G^{\delta}}= (\omega_G^{\delta}\oplus \gamma_G^{\delta})\circ (\nu_G^{\delta})^*,$$
see Theorem \ref{T:altpres}. 

Moreover, the image of $\ov{\gamma_G^{\delta}}$ is equal to 
\begin{equation}\label{E:Pic-rig2}
 \Im(\ov{\gamma_G^{\delta}})=
  \left\{b\in \NS(\Bunr_G^{\delta})\: : \: 
 \begin{aligned}
 & b(\delta\otimes x)+(g-1)b(x\otimes x)   \text{ is a multiple of } \delta(C/S) \\
&  \text{ for any } x\in \Lambda(T_G)
 \end{aligned}
 \right\}.
  \end{equation}
\end{teo}
This Theorem was proved for the universal family $\cC_{g,n}\to \cM_{g,n}$ for $g\geq 1$  in \cite[Thm. 5.5]{FV2}.
\begin{proof}
This is deduced from Theorem \ref{T:altpres}, arguing as in the proofs of \cite[Thm. 5.5]{FV2}.
\end{proof}

We now describe the cokernel of the weight homomorphism. 

\begin{teo}\label{T:cokerwt}
Let $C\to S$ be a family of curves of genus $g>0$ on a regular and integral quotient stack with geometric generic fiber $C_{\ov \eta}$.  Assume that the following two conditions hold true:
\begin{enumerate}[(i)]
\item  $\End(J_{C_{\ov \eta}}) =\bbZ$ and $\RPic(C/S)=\Pic_{C/S}(S)$;
\item either $\RPic^0(C/S)$ is torsion-free or $\scr D(G)$ is simply connected.
\end{enumerate}
Then, for any fixed $\delta\in \pi_1(G)$, we have the following exact sequence,
There exists an exact sequence 
 \begin{equation}\label{E:coker-om2}
 0\to \coker(\ov{\gamma_G^{\delta}})\xrightarrow{\partial_G^{\delta}} \Hom\left(\Lambda(G^{\ab}), \frac{\bbZ}{\delta(C/S)\bbZ}\right)\xrightarrow{\wt{\Lambda_{\ab}^*}} \coker(\w_G^{\delta})\xrightarrow{ \wt{\Lambda_{\scr D}^*}} \coker(\ev_{\scr D(G)}^{\delta})\to 0,
  \end{equation}
where 
$\wt{\Lambda_{\ab}^*}$ and $\wt{\Lambda_{\scr D}^*}$ are the homomorphisms induced by, respectively, $\ov{\Lambda_{\ab}^*}$ and $\ov{\Lambda_{\scr D}^*}$ of \eqref{E:LambdaZG}, and $\partial_G^{\delta}$ is defined as it follows
\begin{equation}\label{E:partial}
\begin{aligned}
\partial_G^{\delta}: \coker(\ov{\gamma_G^{\delta}}) & \longrightarrow \Hom\left(\Lambda(G^{\ab}), \frac{\bbZ}{\delta(C/S)\bbZ}\right)\\
[b] & \mapsto \left\{x\mapsto \left[b(\delta\otimes \wt x)+(1-g) b(\wt x\otimes \wt x)\right]\right\},  
\end{aligned}
\end{equation}
where $b\in \NS(\bgr{G}^{\delta})\subset \Bil^{s,\scr D-\ev}(\Lambda(T_G))^{\scr W_G}$, and $\wt x\in \Lambda(T_G)$ is any lift of $x\in \Lambda(G^{\ab})$. 
\end{teo}
This Theorem was proved for the universal family $\cC_{g,n}\to \cM_{g,n}$ for $g\geq 1$  in \cite[Thm. 5.7]{FV2}
 and for a  curve $C$ over $k=\ov k$ of genus $g(C)\geq 3$  in \cite[Thm. 6.8]{BH12} (and in both results the general case of an arbitrary reductive group is also treated). 
\begin{proof}
This is deduced from Theorem \ref{T:altpres}, arguing as in the proofs of \cite[Thm. 5.7]{FV2}.
\end{proof}

\begin{cor} \label{C:PicRig>0}
Same assumptions as in Theorem \ref{T:cokerwt}. 
\begin{enumerate}
\item \label{C:PicRig>0A} If $\delta(C/S)=1$ then there is an isomorphism 
$$\wt{\Lambda_{\scr D}^*}: \coker(\w_G^{\delta})\xrightarrow{ \cong} \coker(\ev_{\scr D(G)}^{\delta}).$$
\item \label{C:PicRig>0B} If $G=T$ is a torus and $d\in \Lambda(T)$, then  there is an exact sequence 
\begin{equation}\label{E:coker-om3}
 0\to \coker(\ov{\gamma_T^{d}})\xrightarrow{\partial_T^{d}} \Hom\left(\Lambda(T), \frac{\bbZ}{\delta(C/S)\bbZ}\right)\xrightarrow{} \coker(\w_T^{d})\to 0,
  \end{equation}
\end{enumerate}
Moreover, we have that 
 $$
\begin{aligned}
& \coker(\ov{\gamma_T^d})\cong \frac{\bbZ}{\frac{\delta(C/S)}{\gcd(\delta(C/S),\div(d)+1-g)}\bbZ}\oplus \left[ \frac{\bbZ}{\frac{\delta(C/S)}{\gcd(\delta(C/S),g-1,\div(d))}\bbZ}\right]^{\oplus(\dim(T)-1)}  \\
&\coker(\omega_T^d)\cong \frac{\bbZ}{\gcd(\delta(C/S),\div(d)+1-g)\bbZ}\oplus \left[\frac{\bbZ}{\gcd(\delta(C/S),g-1,\div(d))\bbZ}\right]^{\oplus(\dim(T)-1)}, \\
\end{aligned}
$$
where $\div(d)$ is the divisibility of $d$ in the lattice $\Lambda(T)$, with the convention that $\coker(\ov{\gamma_T^d})=\{0\}$ if $\delta(C/S)=0$ (when the above expression for $\coker(\ov{\gamma_T^d})$ is not well-defined).
\end{cor}


The above Corollary for $G=\Gm$ recovers the main result of \cite{MR85}: a Poincar\'e line bundle exists on $\Bunr_{\Gm}^d(C/S)\times_S C$ (which is equivalent to the triviality of the $\Gm$-gerbe $\nu_{\Gm}^d$)  if and only if $\gcd(\delta(C/S),d+1-g)=1$ where $d\in \Lambda(\Gm)=\bbZ$. 

\begin{proof}
Part \eqref{C:PicRig>0A} and the first part of \eqref{C:PicRig>0B} follow from the above Theorem \ref{T:PicRig>0}. 

The second part of \eqref{C:PicRig>0B} follows by choosing a basis $\{e_1,\ldots, e_r\}$ of $\Lambda(T)$ (with $r=\dim T$) such that $d=\div(d)e_1$ and computing the composition 
$$
\ov{\partial_T^d}: \Bil^s(\Lambda(T))\twoheadrightarrow \coker(\ov{\gamma_T^d}) \xrightarrow{\partial_T^d} \Hom\left(\Lambda(T), \frac{\bbZ}{\delta(C/S)\bbZ}\right)
$$ 
on the basis $\left\{\{e_i^*\otimes e_i^*\}_i, \{e_i^*\otimes e_j^*+e_j^*\otimes e_i^*\}_{i>j}\right\}$ of $\Bil^s(\Lambda(T))$:
$$
\begin{sis}
\ov{\partial_T^d}(e_1^*\otimes e_1^*)=[\div(d)+1-g]e_1^*, \\
\ov{\partial_T^d}(e_i^*\otimes e_i^*)=[1-g]e_i^*\: \text{ for any } i\geq 2,\\
\ov{\partial_T^d}(e_1^*\otimes e_i^*+e_i^*\otimes e_1^*)=\div(d) e_i^*\: \text{ for any } i\geq 2,\\
\ov{\partial_T^d}(e_i^*\otimes e_j^*+e_j^*\otimes e_i^*)=0 \: \text{ for any } 2\leq i<j.\\
\end{sis}
$$
\end{proof}

We end this subsection by computing the relative Picard group of $\Bunr_{G}^{\delta}(C/S)$ and the group $\coker(\w_G^{\delta})$ for families of genus $0$, under certain assumptions.

\begin{teo}\label{T:PicRig0}
Let $\pi:C\to S$ is a family of smooth curves of genus  $g(C/S)=0$ over an integral and regular quotient stack $S$. Assume that at least one of the following two assumptions is satisfied:
\begin{enumerate}[(a)]
\item \label{ass-a} the family $C/S$ is Zariski locally trivial;
\item \label{ass-b} the derived subgroup $\scr D(G)$ is simply connected.  
\end{enumerate}
For a given $\delta\in \pi_1(G)$, consider the homomorphism 
\begin{equation}\label{E:ev-dss}
\begin{aligned}
\wh{\ev}_{\scr D(G)}^{\delta}: \left\{\begin{aligned} b\in \left(\Sym^2 \Lambda^*(T_{G^{\sc}})\right)^{\mathscr W_G}: \\ b^{\bbQ}(d^{\ss},-)_{|\Lambda(T_{\scr D(G)})}\text{ is integral }\end{aligned}\right\} & \longrightarrow \frac{\Lambda^*(T_{\scr D(G)})}{\Lambda^*(T_{G^{\ad})}} \\
b &\mapsto b(\delta^{\ss}\otimes -):=[b^{\bbQ}(d^{\ss}\otimes -)], 
\end{aligned}
\end{equation}
where $d^{\ss}\in \Lambda(T_{G^{\ss}})$ is any lift of $\delta^{\ss}:=\pi_1(\ss)(\delta)\in \pi_1(G^{\ss})=\frac{\Lambda(T_{G^{\ss}})}{\Lambda(T_{G^{\sc}})}$ and $d^{\bbQ}$ is the rational extension of $b$ to $\Lambda(T_{G^{\ss}})$.

\begin{enumerate}
\item \label{T:PicRig0a} We have an isomorphism
$$
\RPic(\Bunr_{G}^{\delta}(C/S))\cong \ker(\wh{\ev}_{\scr D(G)}^{\delta}),
$$
which is functorial in $S$ and in $G$.
\item \label{T:PicRig0b} The homomorphism $\ov{\Lambda_{\scr D}^*}$ of \eqref{E:LambdaZG} induces a surjection 
 \begin{equation}\label{E:coker-om1}
 \wt{\Lambda_{\scr D}^*}:\coker(\w_G^{\delta})\twoheadrightarrow \coker(\wh{\ev}_{\scr D(G)}^{\delta}),
 \end{equation}
 whose kernel is equal to 
 \begin{equation}\label{E:ker-coker}
 \ker( \wt{\Lambda_{\scr D}^*})=
 \begin{cases}
 0 & \text{ if either } \delta(C/S)=1 \: \text{ or } \delta^{\ab} \: \text{ is $2$-divisible in } \Lambda^*(G^{\ab}),\\
 \bbZ/2\bbZ  & \text{ if } \delta(C/S)=2 \: \text{ and } \delta^{\ab} \: \text{ is not $2$-divisible in } \Lambda^*(G^{\ab}).\\
 \end{cases}
 \end{equation}
\end{enumerate}
\end{teo}
This Theorem can be deduced for the universal family $\cC_{0,n}\to \cM_{0,n}$ with $n\geq 1$ (which is Zariski locally trivial)  from \cite[Rmk. 5.10]{FV2}. 
Note that if $\scr D(G)$ is simply connected, then the homomorphism $\wh{\ev}_{\scr D(G)}^{\delta}$ coincides with the homomorphism $\ev_{G^{\sc}}^{\delta}$  by Lemma \ref{L:coker-sym}\eqref{L:coker-sym3}.

\begin{proof}
Set 
$$
Q:= \left\{\begin{aligned} b\in \left(\Sym^2 \Lambda^*(T_{G^{\sc}})\right)^{\mathscr W_G}: \\ b^{\bbQ}(d^{\ss},-)_{|\Lambda(T_{\scr D(G)})}\text{ is integral }\end{aligned}\right\}.
$$

Consider the following diagram 
\begin{equation}\label{E:diagwt}
\xymatrix{ 
0\ar[r] &  \RPic \left(\Bun_{G_{\ab}}^{\delta^{\ab}}(C/S)\right)\ar[r]^{\ab_\sharp^*} \ar[d]^{\w_{G^{\ab}}^{\delta^{\ab}}}&  \RPic \left(\Bun_G^{\delta}(C/S)\right)\ar[r] \ar[d]^{\w_{G}^{\delta}}& Q \ar[r] \ar[d]^{\wh{\ev}_{\scr D(G)}^{\delta}} &  0\\
0\ar[r] & \Lambda^*(G^{\ab})\ar[r]^{\ov{\Lambda_{\ab}^*}} & \frac{\Lambda^*(T_G)}{\Lambda^*(T_{G^{\ad}})} \ar[r]^{\ov{\Lambda_{\scr D}^*}}& \frac{\Lambda^*(T_{G^{\sc}})}{\Lambda^*(T_{G^{\ad}})} \ar[r] &  0.
} 
\end{equation}
where the first row is exact by Corollary \ref{C:PicG=0}, the second row is exact by \eqref{E:LambdaZG} and the diagram is commutative by the same proof of \cite[Prop. 5.2, Prop. 5.4]{FV2}.

By applying the snake lemma and using that $\w_{G^{\ab}}^{\delta^{\ab}}$ is injective by Theorem \ref{T:PicBunT-0}\eqref{T:PicBunT-02}, we get an exact sequence
$$
0\to\RPic(\Bunr_{G}^{\delta}(C/S))\to \ker(\wh{\ev}_{\scr D(G)}^{\delta} \xrightarrow{\partial} \coker(\w_{G^{\ab}}^{\delta^{\ab}}) \to \coker(\w_G^{\delta})\xrightarrow{\wt{\Lambda_{\scr D}^*}} \coker(\wh{\ev}_{\scr D(G)}^{\delta})\to 0.
$$
Observe now that Theorem \ref{T:PicBunT-0}\eqref{T:PicBunT-03} implies that $\coker(\w_{G^{\ab}}^{\delta^{\ab}})$ is equal to the group in the right hand side of \eqref{E:ker-coker}. Hence, we conclude by the previous exact sequence and the following 

\un{Claim:}  the coboundary map $\partial$ is zero. 

Indeed,  assume that $\coker(\w_{G^{\ab}}^{\delta^{\ab}})\cong \bbZ/2\bbZ$, for otherwise the statement is trivially true. In particular, $C/S$ is not Zariski locally trivial and hence $\scr D(G)=G^{\sc}$ due to our assumptions. This implies that $\wh{\ev}_{\scr D(G)}^{\delta}=\ev_{G^{\sc}}^{\delta}$ as observed above. 
Let $b\in \ker(\ev_{G^{\sc}}^{\delta})\subset \left(\Sym^2 \Lambda^*(T_{G^{\sc}})\right)^{\mathscr W_G}$. By Lemma \ref{L:coker-sym}\eqref{L:coker-sym1}, we can choose $\wt b\in \left(\Sym^2 \Lambda^*(T_{G})\right)^{\mathscr W_G}$ such that $\wt b_{|\Lambda(T_{G^{\sc}})\times \Lambda(T_{G^{\sc}})}=b$, which implies that the class of $\tau_G^{\delta}(\wt b)\in \RPic \left(\Bun_G^{\delta}(C/S)\right)$  is a lift of $b$. By the same argument of \cite[Def./Lemma 3.7, Prop. 5.2]{FV2}, we have that 
$$
\w_G^{\delta}(\tau_G^{\delta}(\wt b))=\wt b(\delta\otimes -):=[\wt b(d\otimes -)]\in \frac{\Lambda^*(T_G)}{\Lambda^*(T_{G^{\ad}})},
$$
where $d\in \Lambda(T_G)$ is any lift of $\delta$. Since $b$ is in the kernel of $\ev_{G^{\sc}}^{\delta}$, the element $\w_G^{\delta}(\tau_G^{\delta}(\wt b))$ lies in the image of $\ov{\Lambda_{\ab}^*}$, i.e. there exists $\chi\in \Lambda^*(G^{\ab})$ such that 
\begin{equation}\label{E:chi-b}
\wt b(\delta\otimes -)=\chi(\Lambda_{\ab}(-)).
\end{equation}
By the definition of the boundary homomorphism, we have that 
$$
\partial(b)=[\chi(\delta^{\ab})] \in \bbZ/2\bbZ\cong \coker(\w_{G^{\ab}}^{\delta^{\ab}}). 
$$
Since $\delta^{\ab}=\Lambda_{\ab}(d)$ for any lift $d\in \Lambda(T_G)$ of $\delta$, \eqref{E:chi-b} implies that 
$$\chi(\delta^{\ab})=\wt b(\delta\otimes b)=\wt b(d\otimes d).$$
But $\wt b(d\otimes d)$ is even since $\wt b \in  \left(\Sym^2 \Lambda^*(T_{G})\right)^{\mathscr W_G}\cong \Bil^{s,\ev}(\Lambda(T_G))^{\scr W_G}$, which forces $\partial(b)$ to be zero, q.e.d.
\end{proof}

\noindent {\bf Acknowledgments.}

We thank Guido Lido for useful conversations concerning Remark \ref{R:NSjump}.

FV is funded by the MIUR  ``Excellence Department Project'' MATH@TOV, awarded to the Department of Mathematics, University of Rome Tor Vergata, CUP E83C18000100006, by the  PRIN 2022 ``Moduli Spaces and Birational Geometry'' and  the PRIN 2020 ``Curves, Ricci flat Varieties and their Interactions'' funded by MIUR,  by the CMUC (Centro de Matem\'atica da Universidade de Coimbra).

\bibliographystyle{alpha}
\bibliography{BiblioVec}

\begin{thebibliography}{{Sta}18}

\bibitem[AC87]{AC87}
Enrico Arbarello and Maurizio Cornalba.
\newblock The {P}icard groups of the moduli spaces of curves.
\newblock {\em Topology}, 26(2):153--171, 1987.

\bibitem[Amb23]{Amb}
Emiliano Ambrosi.
\newblock Specialization of {N}\'{e}ron-{S}everi groups in positive
  characteristic.
\newblock {\em Ann. Sci. \'{E}c. Norm. Sup\'{e}r. (4)}, 56(3):665--711, 2023.

\bibitem[Bea94]{Bea94}
Arnaud Beauville.
\newblock Vector bundles on {R}iemann surfaces and conformal field theory.
\newblock In {\em R.{C}.{P}. 25, {V}ol. 46 ({F}rench) ({S}trasbourg,
  1992/1994)}, volume 1994/29 of {\em Pr\'{e}publ. Inst. Rech. Math. Av.},
  pages 127--147. Univ. Louis Pasteur, Strasbourg, 1994.

\bibitem[Bea96]{Bea96}
Arnaud Beauville.
\newblock Conformal blocks, fusion rules and the {V}erlinde formula.
\newblock In {\em Proceedings of the {H}irzebruch 65 {C}onference on
  {A}lgebraic {G}eometry ({R}amat {G}an, 1993)}, volume~9 of {\em Israel Math.
  Conf. Proc.}, pages 75--96. Bar-Ilan Univ., Ramat Gan, 1996.

\bibitem[BH10]{BH10}
Indranil Biswas and Norbert Hoffmann.
\newblock The line bundles on moduli stacks of principal bundles on a curve.
\newblock {\em Documenta Math.}, 15:35--72, 2010.

\bibitem[BH12]{BH12}
Indranil Biswas and Norbert Hoffmann.
\newblock Poincar\'e families of {$G$}-bundles on a curve.
\newblock {\em Math. Ann.}, 352(1):133--154, 2012.

\bibitem[BH13]{BHol}
Indranil Biswas and Yogish~I. Holla.
\newblock Brauer group of moduli of principal bundles over a curve.
\newblock {\em J. Reine Angew. Math.}, 677:225--249, 2013.

\bibitem[BK05]{BK05}
Arzu Boysal and Shrawan Kumar.
\newblock Explicit determination of the {P}icard group of moduli spaces of
  semistable {$G$}-bundles on curves.
\newblock {\em Math. Ann.}, 332(4):823--842, 2005.

\bibitem[BL94]{BL}
Arnaud Beauville and Yves Laszlo.
\newblock Conformal blocks and generalized theta functions.
\newblock {\em Comm. Math. Phys.}, 164(2):385--419, 1994.

\bibitem[BLR90]{BLR}
Siegfried Bosch, Werner L\"{u}tkebohmert, and Michel Raynaud.
\newblock {\em N\'{e}ron models}, volume~21 of {\em Ergebnisse der Mathematik
  und ihrer Grenzgebiete (3) [Results in Mathematics and Related Areas (3)]}.
\newblock Springer-Verlag, Berlin, 1990.

\bibitem[BLS98]{BLS98}
Arnaud Beauville, Yves Laszlo, and Christoph Sorger.
\newblock The {P}icard group of the moduli of {$G$}-bundles on a curve.
\newblock {\em Compositio Math.}, 112(2):183--216, 1998.

\bibitem[Chr]{Chr}
A.~Christensen.
\newblock Specialization of n\'eron-severi groups in characteristic $p$.
\newblock Preprint arXiv:1810.06550.

\bibitem[Cil86]{Cil86}
Ciro Ciliberto.
\newblock Rationally determined line bundles on families of curves.
\newblock In {\em The curves seminar at {Q}ueen's, {V}ol. {IV} ({K}ingston,
  {O}nt., 1985--1986)}, volume~76 of {\em Queen's Papers in Pure and Appl.
  Math.}, pages Exp. No. D, 48. Queen's Univ., Kingston, ON, 1986.

\bibitem[Cil87]{Cil87}
Ciro Ciliberto.
\newblock On rationally determined line bundles on a family of projective
  curves with general moduli.
\newblock {\em Duke Math. J.}, 55(4):909--917, 1987.

\bibitem[DN89]{DN}
J.-M. Drezet and M.~S. Narasimhan.
\newblock Groupe de {P}icard des vari\'et\'es de modules de fibr\'es
  semi-stables sur les courbes alg\'ebriques.
\newblock {\em Invent. Math.}, 97(1):53--94, 1989.

\bibitem[Fal94]{Fal94}
Gerd Faltings.
\newblock A proof for the {V}erlinde formula.
\newblock {\em J. Algebraic Geom.}, 3(2):347--374, 1994.

\bibitem[Fal03]{Fa03}
Gerd Faltings.
\newblock Algebraic loop groups and moduli spaces of bundles.
\newblock {\em J. Eur. Math. Soc. (JEMS)}, 5(1):41--68, 2003.

\bibitem[FP19a]{FP16}
Roberto Fringuelli and Roberto Pirisi.
\newblock The {P}icard group of the universal abelian variety and the
  {F}ranchetta conjecture for abelian varieties.
\newblock {\em Michigan Math. J.}, 68(3):651--671, 2019.

\bibitem[FP19b]{FPbr}
Roberto Fringuelli and Roberto Pirisi.
\newblock {The Brauer Group of the Universal Moduli Space of Vector Bundles
  Over Smooth Curves}.
\newblock {\em International Mathematics Research Notices},
  2021(18):13609--13644, 12 2019.

\bibitem[Fri18]{Fri18}
Roberto Fringuelli.
\newblock The {P}icard group of the universal moduli space of vector bundles on
  stable curves.
\newblock {\em Adv. Math.}, 336:477--557, 2018.

\bibitem[FVa]{FV4}
Roberto Fringuelli and Filippo Viviani.
\newblock On the {B}rauer group of the moduli stack of principal bundles over
  families of curves.
\newblock In preparation.

\bibitem[FVb]{FV3}
Roberto Fringuelli and Filippo Viviani.
\newblock On the {P}icard group scheme of the moduli stack of stable pointed
  curves.
\newblock Preprint arXiv:2005.06920.

\bibitem[FV22]{FV1}
Roberto Fringuelli and Filippo Viviani.
\newblock The {P}icard group of the universal moduli stack of principal bundles
  on pointed smooth curves.
\newblock {\em J. Topol.}, 15(4):2065--2142, 2022.

\bibitem[FV23]{FV2}
Roberto Fringuelli and Filippo Viviani.
\newblock The {P}icard group of the universal moduli stack of principal bundles
  on pointed smooth curves {II}.
\newblock {\em Ann. Sc. Norm. Super. Pisa Cl. Sci. (5)}, 24(1):367--447, 2023.

\bibitem[KN97]{KN97}
Shrawan Kumar and M.~S. Narasimhan.
\newblock Picard group of the moduli spaces of {$G$}-bundles.
\newblock {\em Math. Ann.}, 308(1):155--173, 1997.

\bibitem[KNR94]{KNR94}
Shrawan Kumar, M.~S. Narasimhan, and A.~Ramanathan.
\newblock Infinite {G}rassmannians and moduli spaces of {$G$}-bundles.
\newblock {\em Math. Ann.}, 300(1):41--75, 1994.

\bibitem[Koi76]{Koi}
Shoji Koizumi.
\newblock The ring of algebraic correspondences on a generic curve of genus
  {$g$}.
\newblock {\em Nagoya Math. J.}, 60:173--180, 1976.

\bibitem[Kou91]{Kou91}
Alexis Kouvidakis.
\newblock The {P}icard group of the universal {P}icard varieties over the
  moduli space of curves.
\newblock {\em J. Differential Geom.}, 34(3):839--850, 1991.

\bibitem[Kou93]{Kou93}
Alexis Kouvidakis.
\newblock On the moduli space of vector bundles on the fibers of the universal
  curve.
\newblock {\em J. Differential Geom.}, 37(3):505--522, 1993.

\bibitem[Las97]{Laszlo}
Yves Laszlo.
\newblock Linearization of group stack actions and the {P}icard group of the
  moduli of ${SL}_r/\mu_s$-bundles on a curve.
\newblock {\em Bull. Soc. Math. France}, 125(4):529--545, 1997.

\bibitem[LS97]{LS97}
Yves Laszlo and Christoph Sorger.
\newblock The line bundles on the moduli of parabolic {$G$}-bundles over curves
  and their sections.
\newblock {\em Ann. Sci. \'Ecole Norm. Sup. (4)}, 30(4):499--525, 1997.

\bibitem[Mes87]{Mes}
Nicole Mestrano.
\newblock Conjecture de {F}ranchetta forte.
\newblock {\em Invent. Math.}, 87(2):365--376, 1987.

\bibitem[Mor76]{Mo76}
Shigefumi Mori.
\newblock The endomorphism rings of some {A}belian varieties.
\newblock {\em Japan. J. Math. (N.S.)}, 2(1):109--130, 1976.

\bibitem[MP12]{MP}
Davesh Maulik and Bjorn Poonen.
\newblock N\'{e}ron-{S}everi groups under specialization.
\newblock {\em Duke Math. J.}, 161(11):2167--2206, 2012.

\bibitem[MR85]{MR85}
N.~Mestrano and S.~Ramanan.
\newblock Poincar\'e bundles for families of curves.
\newblock {\em J. Reine Angew. Math.}, 362:169--178, 1985.

\bibitem[MV14]{MV}
Margarida Melo and Filippo Viviani.
\newblock The {P}icard group of the compactified universal {J}acobian.
\newblock {\em Doc. Math.}, 19:457--507, 2014.

\bibitem[Pau96]{Pau96}
Christian Pauly.
\newblock Espaces de modules de fibr\'{e}s paraboliques et blocs conformes.
\newblock {\em Duke Math. J.}, 84(1):217--235, 1996.

\bibitem[Pir88]{Piro}
Gian~Pietro Pirola.
\newblock Base number theorem for abelian varieties. {A}n infinitesimal
  approach.
\newblock {\em Math. Ann.}, 282(3):361--368, 1988.

\bibitem[Ray70]{Ray}
Michel Raynaud.
\newblock {\em Faisceaux amples sur les sch\'{e}mas en groupes et les espaces
  homog\`enes}.
\newblock Lecture Notes in Mathematics, Vol. 119. Springer-Verlag, Berlin-New
  York, 1970.

\bibitem[Sch03]{Sch03}
Stefan Schr{\"o}er.
\newblock The strong {F}ranchetta conjecture in arbitrary characteristics.
\newblock {\em Internat. J. Math.}, 14(4):371--396, 2003.

\bibitem[Sor96]{Sor96}
Christoph Sorger.
\newblock La formule de {V}erlinde.
\newblock Number 237, pages Exp. No. 794, 3, 87--114. 1996.
\newblock S\'{e}minaire Bourbaki, Vol. 1994/95.

\bibitem[Sor99]{SO99}
Christoph Sorger.
\newblock On moduli of {$G$}-bundles of a curve for exceptional {$G$}.
\newblock {\em Ann. Sci. \'Ecole Norm. Sup. (4)}, 32(1):127--133, 1999.

\bibitem[{Sta}18]{stacks-project}
The {Stacks Project Authors}.
\newblock \textit{Stacks Project}.
\newblock \url{https://stacks.math.columbia.edu}, 2018.

\bibitem[Tel98]{Tel98}
Constantin Teleman.
\newblock Borel-{W}eil-{B}ott theory on the moduli stack of {$G$}-bundles over
  a curve.
\newblock {\em Invent. Math.}, 134(1):1--57, 1998.

\bibitem[Zar28]{Zar}
Oscar Zariski.
\newblock On a {T}heorem of {S}everi.
\newblock {\em Amer. J. Math.}, 50(1):87--92, 1928.

\bibitem[Zom]{Zom}
Wouter Zomervrucht.
\newblock {\em Ineffective descent of genus one curves}.
\newblock Available at http://arxiv.org/abs/1501.04304.

\end{thebibliography}
\end{document}